\documentclass[a4paper,11pt,reqno]{amsart}
\usepackage{latexsym,amscd,amssymb,url}
\usepackage{comment}
\usepackage{extarrows} 
\usepackage{mathtools}
\usepackage{tikz-cd}

\usepackage[all,cmtip]{xy}  
\usepackage{graphicx} 
\usepackage{xcolor} 
\usepackage{enumitem} 
\definecolor{dblue}{rgb}{0,0,.6}
\usepackage[colorlinks=true, linkcolor=dblue, citecolor=dblue, filecolor = dblue, menucolor = dblue, urlcolor = dblue]{hyperref}

\numberwithin{equation}{section}

\newtheorem{theorem}{Theorem}[section]

\theoremstyle{plain}

\newtheorem{claim}[theorem]{Claim}

\newtheorem{conjecture}[theorem]{Conjecture}
\newtheorem{corollary}[theorem]{Corollary}
\newtheorem{cor}[theorem]{Corollary}

\newtheorem{definition}[theorem]{Definition}
\newtheorem{example}[theorem]{Example}

\newtheorem{lemma}[theorem]{Lemma}
\newtheorem{notation}[theorem]{Notation}

\newtheorem{proposition}[theorem]{Proposition}
\newtheorem{prop}[theorem]{Proposition}
\newtheorem{remark}[theorem]{Remark}

\newcommand{\Z}{\mathbb Z}
\newcommand{\Q}{\mathbb Q}
\newcommand{\A}{\mathbb A}

\newcommand{\C}{\mathbb C}

\newcommand{\CP}{\mathbb P}

\newcommand{\F}{\mathbb F}

\newcommand{\im}{\operatorname{im}}

\newcommand{\Pic}{\operatorname{Pic}}

\newcommand{\id}{\operatorname{id}}

\newcommand{\Spec}{\operatorname{Spec}}

\newcommand{\codim}{\operatorname{codim}}

\newcommand{\CH}{\operatorname{CH}}

\newcommand{\cl}{\operatorname{cl}}

\newcommand{\et}{\text{\'et}}
\newcommand{\proet}{\text{pro\'et}}

\newcommand{\identity}{\operatorname{id}}
\newcommand{\image}{\operatorname{im}}

\newcommand{\dashedlongrightarrow}{\xymatrix@1@=15pt{\ar@{-->}[r]&}}
\renewcommand{\longrightarrow}{\xymatrix@1@=15pt{\ar[r]&}}
\renewcommand{\mapsto}{\xymatrix@1@=15pt{\ar@{|->}[r]&}}
\renewcommand{\twoheadrightarrow}{\xymatrix@1@=15pt{\ar@{->>}[r]&}}
\newcommand{\hooklongrightarrow}{\xymatrix@1@=15pt{\ar@{^(->}[r]&}}
\newcommand{\congpf}{\xymatrix@1@=15pt{\ar[r]^-\sim&}}
\renewcommand{\cong}{\simeq}

\newcommand{\kbar}{\bar{\mathbb F}}

\newcommand{\Xbar}{\bar{X}}

\newcommand{\zl}{\mathbb{Z}_{\ell}}
\newcommand{\ql}{\mathbb{Q}_{\ell}}
\newcommand{\rou}[2]{\mu^{\otimes#2}_{#1}}
\newcommand{\coh}[3]{H^{#1}(#2,#3)}

\newcommand{\cohbm}[3]{H_{BM}^{#1}(#2,#3)}

\newcommand{\nrcoh}[4]{H^{#1}(F_{#2}#3,#4)}
\newcommand{\urcoh}[4]{H^{#1}_{#2,nr}(#3,#4)}

\newcommand{\chow}[2]{\operatorname{CH}^{#1}(#2)}

\newcommand{\bigtate}[3]{\mathbf{T}_{#3}^{#1}(#2)}

\newcommand{\bigSS}[3]{\mathbf{SS}_{#3}^{#1}(#2)}
\newcommand{\bigbeil}[3]{\mathbf{B}_{#3}^{#1}(#2)}

\begin{document}    
\title[Cycle conjectures and birational invariants over finite fields]{Cycle conjectures and birational invariants over finite fields}

\author{Samet Balkan} 
\address{Institute of Algebraic Geometry, Leibniz University Hannover, Welfengarten 1, 30167 Hannover, Germany.}
\email{balkan@math.uni-hannover.de}

\author{Stefan Schreieder} 
\address{Institute of Algebraic Geometry, Leibniz University Hannover, Welfengarten 1, 30167 Hannover, Germany.}
\email{schreieder@math.uni-hannover.de}

\date{November 6, 2025} 
%



\begin{abstract}  
We study a natural birational invariant for varieties over finite fields and show that its vanishing on projective space is equivalent to 
 the Tate conjecture, the Beilinson conjecture, and the Grothendieck--Serre semi-simplicity conjecture for all smooth projective varieties over finite fields. 
We further show that the Tate, Beilinson, and 1-semi-simplicity conjecture in half of the degrees implies those conjectures in all degrees.
\end{abstract}
 
\maketitle 
 

\section{Introduction}

Throughout this paper $\F$ denotes a finite field of characteristic $p$, $\bar \F$ denotes an algebraic closure of $\F$ and $G_{\F}$ denotes the absolute Galois group of $\F$.
For an $\F$-scheme $X$, we denote by $\bar X$ the base change of $X$ to $\bar \F$.
We further fix a prime $\ell\neq p$. 

\subsection{Cycle conjectures}
Let $X$ be a smooth projective variety over $\F$.
The Galois group $G_\F$ acts on the $\ell$-adic \'etale cohomology $H^{i}(\bar X,\Q_\ell(j))$ of $\bar X$ and the Grothendieck--Serre semi-simplicity conjecture predicts that this action is semi-simple, see \cite[Conj. 12.5]{jannsen-LMN}; in other words, the Frobenius action on $H^{i}(\bar X,\Q_\ell(j)) $ is diagonalizable (over $\bar \Q_\ell$).

There is a cycle class map
\begin{align} \label{eq:cl}
\cl_X^i:\CH^i(X)_{\Q_\ell}\longrightarrow H^{2i}(\bar X,\Q_\ell(i))
\end{align}
which lands in the subspace of Galois invariant classes  $H^{2i}(\bar X,\Q_\ell(i))^{G_\F}$.
The Tate conjecture asserts that the image of $\cl_X^i$ coincides with that subspace.
This is known for divisors on abelian varieties \cite{tate} and for K3 surfaces \cite{artin-swinnerton-dyer,rudakov-shafarevich-zink,nygaard-ogus,maulik, charles,madapusipera,kim-madapusipera}, but wide open in general, see e.g.\ the surveys \cite{tate2,totaro-tate}.
For important implications and equivalent formulations, see e.g.\ \cite[p.\ 578-579]{totaro-tate}, \cite{geisser98}, and  \cite{kahn98}.
We also note that some reduction steps are known.
For instance, de Jong and Morrow showed that the Tate conjecture for surfaces over $\F$ implies the Tate conjecture for divisors on any smooth projective variety over $\F$, see \cite{morrow}. 


The Beilinson conjecture predicts that the cycle class map in \eqref{eq:cl} is always injective.
This is known for $i=0,1,\dim X$, but wide open in general; some interesting special cases are proven in \cite{soule,kahn}. 
(We follow here the terminology from \cite{totaro-tate}; Beilinson's original conjecture is slightly stronger, as it asserts that any cycle in $\CH^i(X)_{\Q_\ell}$ that is numerically trivial is already zero in $\CH^i(X)_{\Q_\ell}$.)

The  Tate and Beilinson conjectures are fundamental open problems.
They are analogues over finite fields of the Hodge and the Bloch--Beilinson conjecture over the complex numbers, while the semi-simplicity conjecture is an analogue of the well-known fact that the category of polarized rational Hodge structures is semi-simple.

\subsection{Birational invariants} 
For a smooth projective $ \F$-variety $X$, we denote by
\begin{align} \label{eq:H^iF0X}
H^i(F_0\bar X,\Q_\ell(n)):=\lim_{\substack{\longrightarrow\\ \emptyset\neq  \bar U\subset \bar X}} H^i(\bar U,\Q_\ell(n))
\end{align}
the direct limit of the \'etale cohomology groups of all dense open subsets $\bar U$ of $\bar X$ with coefficients in $\Q_\ell(n)$.
If $\bar X$ is irreducible, this group  is a birational invariant of $\bar X$, which depends only on the function field $\bar \F(X)$ and which plays an important role in the Gersten conjecture, see e.g.\ \cite{BO}. 
However, the group is too big for practical purposes and so one often considers suitable subgroups, for instance the subgroup $H^i_{nr}(\bar X,\Q_\ell(n))\subset H^i(F_0\bar X,\Q_\ell(n))$ of unramified classes, see \cite{CTO}.
The latter plays a central role in rationality problems, (where one uses torsion coefficients instead of $\Q_\ell(n)$), see e.g.\ \cite{CTO,HPT,Sch-JAMS} and the surveys \cite{CT-survey,Sch-survey}.

Since $X$ is defined over $\F$ and the limit in \eqref{eq:H^iF0X} can be computed by running only over those open dense subsets of $\bar X$ that are defined over $\F$,  the group \eqref{eq:H^iF0X} inherits a natural Galois action.
In analogy to the idea of considering the subgroup of unramified classes, it is  natural to consider the subgroup
$$
H^i(F_0\bar X,\Q_\ell(n))^{G_\F}\subset H^i(F_0\bar X,\Q_\ell(n))
$$
of Galois invariant classes.
This is a birational invariant of $X$, hence an invariant of the function field $\F(X)$.
It is natural to expect that this invariant should vanish for $X=\CP^m$ for suitable values of $i$ and $n$.
For instance, we will show in this paper (see Theorem \ref{thm:induct1}\eqref{induct1_1} below) that
$$
H^{2i}(F_0 \CP^n_{\bar \F}, \Q_\ell(i))^{G_{\F}}=0\ \ \text{for all}\ \ i,n\geq 2 .
$$  

\subsection{Main results}
Our first main result is as follows.

\begin{theorem}  \label{thm:main:Tate-Beilinson-ss}
The Tate,   Beilinson, and semi-simplicity conjecture for all smooth projective varieties over a finite field $\F$ are equivalent to the vanishing
$$
H^{2i}(F_0 \CP^n_{\bar \F }, \Q_\ell(i+1))^{G_{\F}}=0\ \ \text{for all}\ \ i,n \geq 2 .
$$   
\end{theorem}

By the above theorem, the vanishing of a natural birational invariant of projective space is equivalent to the  Tate,   Beilinson, and semi-simplicity conjecture for all smooth projective varieties over a finite field.

In Theorem \ref{thm:cohcrit1} below we give effective versions of the above theorem.
For instance, 
$$
H^4(F_0 \CP^4_{\bar \F }, \Q_\ell(3))^{G_{\F}}=0 
$$
is equivalent to the Tate conjecture for divisors on surfaces and the Beilinson conjecture for threefolds over $\F$,  see  Remark \ref{rem:effective-version} below.

A variant of Theorem \ref{thm:main:Tate-Beilinson-ss} that concerns only the Tate and semi-simplicity conjecture is as follows.

\begin{theorem}  \label{thm:main:Tate-ss}
The Tate and semi-simplicity conjecture for all smooth projective varieties over a finite field $\F$ are equivalent to the vanishing
$$
H^{2i-1}(F_0 \CP^n_{\bar \F }, \Q_\ell(i))^{G_{\F}}=0\ \ \text{for all}\ \ i,n \geq 2 .
$$   
\end{theorem}

As before, the above theorem admits effective versions, see Theorem \ref{thm:cohcrit1}.
For instance, 
 $$
H^3(F_0 \CP^3_{\bar \F }, \Q_\ell(2))^{G_{\F}}=0 
$$
is equivalent to the Tate conjecture for surfaces over $\F$, see  Remark \ref{rem:effective-version} below.  

Our proof of Theorem \ref{thm:main:Tate-Beilinson-ss} leads to the following application of independent interest.

\begin{corollary}  \label{cor:Lefschetz} 
  If for all smooth projective varieties $X$ of dimension $d\leq n$ over a finite field $\F$,  the Tate, Beilinson, and 1-semi-simplicity conjectures hold for cycles of codimension $i\leq \lceil d/2\rceil $, then they hold for cycles of any codimension on smooth projective varieties of dimension at most $n$ over $\F$.
\end{corollary}

The above corollary says that one roughly speaking has to prove only half of the cases of the Tate, Beilinson and 1-semi-simplicity conjecture.
For the Tate and 1-semi-simplicity conjecture, this is a direct consequence of the hard Lefschetz theorem, proven by Deligne.
The new content of the above theorem is the respective assertion for Beilinson's conjecture.
Of course this would follow if the hard Lefschetz conjecture for Chow groups over finite fields was known, i.e.\ if intersection with the relevant power of the hyperplane class would be known to induce isomorphisms $\CH^i(X)_{\Q_\ell}\stackrel{\cong}\to \CH^{\dim X-i}(X)_{\Q_\ell}$, which itself is a consequence of the Tate and Beilinson conjectures.
However, the Lefschetz conjecture for Chow groups over finite fields is known only in some special cases, e.g.\ in the case of abelian varieties (see \cite{soule}),  and wide open in general. 

\subsection{Tate conjecture for higher Chow groups}

To put the above theorem in perspective,  we note that the Tate conjecture  for higher Chow groups (of open varieties) says that the natural map
$$
\CH^j(X,2j-i)\otimes_\Z \Q_\ell \cong H^i_M(X,\Q(j))\otimes_\Q \Q_\ell \longrightarrow H^i(\bar X,\Q_\ell(j))^{G_\F}
$$
from Bloch's higher Chow groups with $\Q_\ell$-coefficients (resp.\ motivic cohomology) to the $G_\F$-invariant subspace of \'etale cohomology is surjective for all smooth varieties $X$ over $\F$.
The Beilinson conjecture for higher Chow groups predicts that the above map is  injective.
The conjunction of both conjectures has been formulated for instance by Beilinson,  Friedlander, and Jannsen, see e.g.\ \cite[Conjecture 12.4(b)]{jannsen-LMN}. 

Since $\CH^a(\Spec K,b)=0$ for any field $K$ and $b<a$, and because motivic cohomology commutes with projective systems of schemes with affine dominant transition maps (see e.g.\ \cite[Example 11.1.25]{cisinski-deglise}),   
we see that the Tate conjecture for higher Chow groups predicts the following for function fields:

\begin{conjecture}[Tate conjecture for function fields]
Let $X$ be a geometrically integral variety over a finite field $\F$. Then the natural map $$ 
\CH^j(\Spec \bar  \F(X),2j-i)\otimes_{\Q}\Q_\ell 
\longrightarrow H^i(F_0\bar X,\Q_\ell(j))^{G_\F}
$$
is surjective.
In particular, $H^i(F_0\bar X,\Q_\ell(j))^{G_\F}=0$ for  $i>j$.
\end{conjecture}
 
In view of this,  Theorem \ref{thm:main:Tate-Beilinson-ss} shows the following:

\begin{corollary}\label{cor:generalized Tate}
Let $\F$ be a finite field.
The Tate conjecture for higher Chow groups applied to the purely transcendental extension $\F(x_1,\dots ,x_n)$ for all $n$
implies the Tate, Beilinson, and semi-simplicity conjecture for all smooth projective varieties over $\F$.
\end{corollary}

A result of Moonen \cite{moonen} says that in characteristic zero, the usual Tate conjecture implies the semi-simplicity conjecture.
The analogous result over finite fields is known under the additional assumption  
of one of Grothendieck's standard conjectures, and goes back to Milne \cite{milne-AJM}.
Corollary \ref{cor:generalized Tate} shows that the Tate conjecture for higher Chow groups implies not only the semi-simplicity conjecture but also the Beilinson conjecture over finite fields, without reference to any standard conjecture. 

\subsection{A conjectural description of $H^i(F_0\bar X,\Q_\ell(n))^{G_\F}$.}
 A conjecture of Beilinson and Parshin predicts that for any field $K$ of positive characteristic, we have $H^i_M(\Spec K,\Q(j))=0$  for $i\neq j$, see \cite[Conjecture 8.3.3(a)]{beilinson-absolute-hodge}.
 Taken together with the Tate and Beilinson conjectures for higher Chow groups, one arrives at the following conjectural description,
see Remark \ref{rem:beilinson-conjecture} below for more details.
 
 \begin{conjecture}
 Let $X$ be a quasi-projective variety over a finite field $\F$. Then
$$
H^i(F_0 \bar X,\Q_\ell(j))^{G_\F}\cong \begin{cases} K^M_i(\F(X))\otimes_{\Z} \Q_\ell\ \ \ \ &\text{if $i=j$};\\
0\ \ \ \ &\text{otherwise}.
\end{cases}
$$
\end{conjecture}

In other words, the birational invariant $H^i(F_0\bar X,\Q_\ell(j))^{G_\F}$ of $X$ that we study in this paper agrees conjecturally with Milnor K-theory (tensored with $\Q_\ell$) for $i=j$ and vanishes otherwise.
If one could prove this for the function field of projective space (and in degree $i=2m$ and weight $j=m+1$),  then, by Theorem \ref{thm:main:Tate-Beilinson-ss}, the Tate, Beilinson, and Grothendieck--Serre semi-simplicity conjecture follow.

\subsection{Organization of the paper}
We collect Preliminaries in Section \ref{sec:preliminary}, discuss Galois actions  in Section \ref{sec:galois} and various forms of the Tate, Beilinson and semi-simplicity conjecture in Section \ref{sec:cycle-conj}.
Most of those results are easy or known to experts; we give some details for convenience of the reader and to fix notation.
In Section \ref{sec:more-ss} we briefly discuss some specific aspects of the $1$-semi-simplicity conjecture and the lifting of Galois invariant classes, that we will use in our inductive arguments.
The proofs of our main results then start in Section \ref{sec:galois-invariants}, where we relate various cycle conjectures for a smooth projective variety $X$ to the Galois invariants of the cohomology of its function field.
The proofs are then concluded in Section \ref{sec:reduction-to-Pn}, where we pass from the function field of a smooth projective variety $X$ of dimension $n$ to the function field of $\CP^{n+1}$, exploiting the fact that any such $X$ is birational to a hypersurface in $\CP^{n+1}$.
 
\section{Preliminaries} \label{sec:preliminary}
\subsection{Conventions}
We fix throughout a finite field $\mathbb F$ of characteristic $p$ and with $q$ elements.

A field is finitely generated if it is finitely generated over its prime field.
For a field $k$ with a separable closure $k_s$, $G_k\coloneqq Gal(k_s/k)$ denotes the absolute Galois group of $k$.  
An algebraic scheme is a separated scheme of finite type over a field. 
A variety is an integral algebraic scheme.

For an equi-dimensional Noetherian scheme $X$, we define $X^{(i)}$ as  the set of all codimension-$i$ points of $X$, i.e.\ $X^{(i)}\coloneqq\{x\in X\big|\dim X-\dim\overline{\{x\}}=i\}$.
We further denote by 
$$
F_iX\coloneqq\{x\in X\big|\dim X-\dim\overline{\{x\}}\leq i\}
$$
the set of points of codimension at most $i$.
This yields a filtration $F_0X\subset F_1X\subset \dots \subset X$ and we may regard $F_jX$ as a pro-scheme given by the system of all open subsets of $X$ that contain $X^{(j)}$.

We denote by
$\chow{i}{X}$   the Chow group of codimension-$i$ cycles of $X$.
For a ring $R$, we denote by $\chow{i}{X}_{R}$ the tensor product $\chow{i}{X}\otimes_{\mathbb{Z}} R$.

Let $\Xbar$ be a quasi-projective variety over an algebraically closed field $\bar k$. 
Let $k$ be a finitely generated field contained in $\bar k$.
 A $k$-variety $X$ is a model of $\bar X$ if $\bar X\simeq X\times_k\bar k$. Whenever such a model exists, we will also say that $\Xbar$ is defined over $k$.

Let $G$ be a group and $M$ be a $G$-module. We write 
$$M^{G}\coloneqq\{x\in M\ \big|\ gx=x\:\forall g\in G\}\ \ \ \text{and}\ \ \ M_{G}\coloneqq M/\{x-gx\ \big|\ g\in G, x\in M\}$$
for the $G$-invariance and $G$-coinvariance of $M$, respectively.

Let $X$ be a variety over a field $k$. An alteration of $X$ is a dominant proper morphism $X^{\prime}\rightarrow X$ of varieties over $k$ such that $\dim X=\dim X^{\prime}$ and $X^\prime$ is regular.
Alterations exist by the work of de Jong, see \cite{deJong}.

\subsection{\'Etale cohomology}
For an algebraic scheme $X$ over a field $k$ and for $A\in \{\Z/\ell^r,\Z_\ell,\Q_\ell\}$, we denote by 
$$
H^i(X,A(n)):=H^i(X_{\proet},A(n))\cong H^i_{cont}(X_\et,A(n))
$$
the pro-\'etale cohomology of $X$ with coefficients in $A$ with a Tate twist of Bhatt and Scholze  \cite{BS},  where $\Z/\ell^r(n):=\mu_{\ell^r}^{\otimes n}$.
These groups agree with Jannsen's continuous \'etale cohomology \cite{jannsen}, cf.\ \cite[\S 5.6]{BS}.  
If $k$ is finite or algebraically closed, then $H^i(X,\Z/\ell^r(n)) $ is finite and so the above groups coincide with ordinary \'etale cohomology, see \cite[p.\ 208]{jannsen}.
In this paper the case $A=\Q_\ell$ will play a particularly important role and we use the shorthand
$$
H^i(X,n):=H^i(X_{\et},\Q_\ell(n)) .
$$

\subsection{Hochschild-Serre spectral sequence} 
Let $X$ be an algebraic scheme over the finite field $\F$.
Since $\F$ has cohomological dimension $1$, the Hochschild--Serre spectral sequence degenerates at $E_2$, see \cite{jannsen,milne}.
Moreover,  for any continuous $G_\F$-module $M$, there is a canonical isomorphism
$$
H^1(G_\F,M)\cong M_{G_{\F}}.
$$
As a consequence,  one can deduce for $A\in \{\Z/\ell^r,\Z_\ell,\Q_\ell\}$ the following well-known short exact sequence 
\begin{equation}\label{HSexact}
0\rightarrow\coh{i-1}{\Xbar}{A(j)}_{G_\F}\rightarrow\coh{i}{X}{A(j)}\rightarrow\coh{i}{\Xbar}{A(j)}^{G_\F}\rightarrow 0 .
\end{equation}

\subsection{Borel-Moore cohomology}\label{section_Borel_Moore}
We will also use Borel--Moore cohomology, defined by
$$
H^{i}_{BM}(X, A(n)):=
H^{i-2d_X} (X_{\proet},\pi_X^!A(n-d_X)), 
$$
where $\pi_X:X\to \Spec k$ is the structure map, $d_X=\dim X$, and $X_\proet$ denotes the pro-\'etale site of $X$, see \cite{BS} and \cite[Section 4 and Proposition 6.6]{Sch-refined}.
If $k$ is a finite field or the algebraic closure of a finite field, then the pro-\'etale site may be replaced by the \'etale site in the above formula, as in this case the above groups are finite $A$-modules. 

\begin{remark}
By definition, the above group coincides with Borel--Moore homology up to a change of indices:
$H^{i}_{BM}(X, A(n))=H^{BM}_{2d_X-i}(X,A(d_X-n))$.
We could of course use the (more standard) homological notation in what follows.
An advantage of our convention is that $H^{i}_{BM}(X, A(n))$ agrees for $X$ smooth and equi-dimensional with ordinary cohomology (because $\pi_X^!\cong \pi_X^\ast(d_X)[2d_X]$ in this case), see Lemma \ref{bmtoet} below.
Since the Tate conjecture (for smooth projective varieties) as well as our main results are formulated cohomologically, a change of indices by going back and forth between homological and cohomological notation is avoided by our convention. 
Also, formulas such as weight considerations (see e.g.\ Lemma \ref{BMweight} below) literally agree with the corresponding formulas in the smooth case.
\end{remark}

Let $X$ be an algebraic scheme. 
We list a few properties of the above functor in what follows, see e.g.\ \cite[Definition 4.2 and Proposition 6.6]{Sch-refined}.
A proper morphism $f:X\rightarrow Y$ of algebraic schemes of relative dimension $c=\dim Y-\dim X$ induces functorial push-forward morphisms
$$
f_*: \cohbm{i-2c}{X}{A(n-c)}\rightarrow\cohbm{i}{Y}{A(n)}
$$
and an open immersion $j:U\hookrightarrow X$ of algebraic schemes with $\dim U=\dim X$ induces functorial pull-back morphisms
$$
j^*:\cohbm{i}{X}{A(n)}\rightarrow\cohbm{i}{U}{A(n)}.
$$ 
For any closed immersion $i:Z\hookrightarrow X$ of codimension $c=\dim X-\dim Z$ with complement $j:U\hookrightarrow X$ such that $\dim U=\dim X$, there exists a long exact Gysin sequence
\begin{equation}\label{Gysin}
...\rightarrow\cohbm{i}{X}{A(n)}\xrightarrow{j^*}\cohbm{i}{U}{A(n)}\xrightarrow{\partial}\cohbm{i+1-2c}{Z}{A(n-c)}\xrightarrow{i_*}\cohbm{i+1}{X}{A(n)}\rightarrow...   
\end{equation}
where $j^*$ is the pull-back and $i_*$ is the push-forward from above. The boundary map $\partial$ is called the residue map.

\begin{lemma}\label{bmtoet}
Let $X$ be a smooth and equi-dimensional algebraic scheme over a field $k$. 
Then, there exist canonical isomorphisms
$$
\coh{i}{X}{A(n)}\xrightarrow{\simeq}\cohbm{i}{X}{A(n)},
$$
where $A(n)\in\{\rou{\ell^r}{n},\zl(n),\ql(n)\}$. 
Moreover, if $f:X\to Y$ is a proper generically finite morphism between smooth varieties,  then
\begin{align} \label{eq:f_*f^*}
f_\ast \circ f^\ast= d\cdot \id:
H^i(Y,A(n)) \stackrel{f^\ast}\longrightarrow H^i(X,A(n))\cong H^i_{BM}(X,A(n))  \stackrel{f_\ast}\longrightarrow H^i(Y,A(n)) ,
\end{align}
where $d$ denotes the degree of $f$ (i.e.\ the degree of the field extension $k(X)/k(Y)$).
\end{lemma}
\begin{proof}
The first claim follows from Poincar\'e duality, see e.g.\  \cite[Lemma~6.5]{Sch-refined}.
For the second claim, note that $f_\ast [X]=d\cdot [Y]\in \CH^0(Y)$, by definition of proper pushforwards of algebraic cycles.
Cup product with the cycle class $\cl(Y)\in H^0(Y,A(0))$ on $Y$ agrees with the identity and so \eqref{eq:f_*f^*} follows from   the projection formula $f_\ast(f^\ast \alpha \cup \beta)=\alpha\cup f_\ast \beta$,   see \cite[Lemma A.19(1)]{Sch-moving}, and the compatibility of $f^\ast$ and $f_\ast$ with the cycle class map, see \cite[Lemma A.21]{Sch-moving}. 
\end{proof}

The following lemma follows easily from the existence of the Gysin sequence \eqref{Gysin}.  
\begin{lemma}\label{directsum}
Let $X$ be an algebraic scheme.
\begin{enumerate}
    \item  Let $X_1$,...,$X_n$ be  the connected components of $X$. For $m=1,...,n$, let $c_m\coloneqq \dim X-\dim X_m$ be the codimension of the component $X_m$. Then, for all $i$ and $j$, 
    $$\cohbm{i}{X}{A(j)}\simeq\underset{m=1,...,n}{\bigoplus}\cohbm{i-2c_m}{X_m}{A(j-c_m)}.$$
    \item Let $X^{red}$ be the scheme $X$ with its reduced scheme structure. Then, for all $i$ and $j$, there is an isomorphism
    $$\cohbm{i}{X}{A(j)}\simeq\cohbm{i}{X^{red}}{A(j)}.$$
\end{enumerate}
\end{lemma}

\begin{lemma}\label{negdegree}
Let $X$ be an algebraic scheme over a finite or algebraically closed field $k$. Let $\ell$ be a prime different from the characteristic of $k$. Then, for all $i$ and $j$, $\cohbm{i}{X}{\rou{\ell^r}{j}}$ is a finite group and $\cohbm{i}{X}{\zl(j)}$ is a finitely generated $\zl$-module  for all $i$ and $j$.
Moreover,  $\cohbm{i}{X}{\zl(j)}$ and $\cohbm{i}{X}{\rou{\ell^r}{j}}$ vanish for $i<0$.
\begin{proof}
Using the Gysin sequence (\ref{Gysin}) and Lemma \ref{bmtoet}, one can reduce the problem via an induction on the dimension to the exact same statement for the \'etale cohomology. Then, the vanishing in negative degrees becomes trivial. As for finiteness, it follows from \cite[Corollaire~1.10]{SGA4.5} and \cite[Lemma~1.11 p. 165]{milne} that $\cohbm{i}{X}{\rou{\ell^r}{j}}$ is a finite and $\cohbm{i}{X}{\zl(j)}$ is a finitely generated $\Z_\ell$-module if $k$ is algebraically closed. The finiteness for algebraically closed fields imply the finiteness for finite fields by the Hochschild-Serre exact sequence \eqref{HSexact}.
\end{proof}
\end{lemma}

\subsection{Refined unramified cohomology}\label{subsec:H_j,nr}

For an algebraic scheme $X$, let 
$$
F_jX\coloneqq\{x\in X\big| \codim (x)\leq j\},
$$
where $\codim (x)=\dim X-\dim\overline{\{x\}}$ and let $X^{(j)}\coloneqq\{x\in X\big| \codim (x)=j\}$. 
The following definitions are taken from \cite[Section~5]{Sch-refined}.
\begin{definition}\label{unramcoh}
Let $X$ be an algebraic scheme. For all integers $i$, $j$ and $n$,
$$\nrcoh{i}{j}{X}{A(n)}\coloneqq\underset{F_j X\subseteq U,\:\text{open}}{\varinjlim}\cohbm{i}{U}{A(n)},$$
where the direct limit runs through all the open sets $U$ of $X$ such that $F_j X\subseteq U$. If $X$ is irreducible with the generic point $\eta$, we also write $\coh{i}{\eta}{A(n)}$ instead of $\nrcoh{i}{0}{X}{A(n)}$.\par
For all $m$ with $m\geq j$, there are natural restriction maps $\nrcoh{i}{m}{X}{A(n)}\rightarrow\nrcoh{i}{j}{X}{A(n)}$, which are induced by pull-backs via open immersions. The $j$-th refined unramified cohomology group of $X$ is defined as
$$\urcoh{i}{j}{X}{A(n)}\coloneqq \image \big(\nrcoh{i}{j+1}{X}{A(n)}\rightarrow\nrcoh{i}{j}{X}{A(n)}\big).$$
\end{definition}

Refined unramified cohomology generalizes the concept of unramified cohomology, which corresponds to the case $j=0$, see \cite{Sch-refined}.

\begin{remark}\label{Remark_after_urcoh}
Let $X$ be an algebraic scheme.
\begin{enumerate}
\item If $X$ is smooth and equi-dimensional, then the same holds for the open subsets of $X$ and so Borel--Moore cohomology can in this case be replaced by ordinary cohomology in Definition \ref{unramcoh}.
    \item $F_jX$ may be thought of as the pro-scheme of all open subsets of $X$ that contain all codimension $j$ points.
    The group $\nrcoh{i}{j}{X}{A(n)}$ coincides with the cohomology of the pro-scheme $F_jX$. 
    In particular,  if $X$ is integral, $\nrcoh{i}{0}{X}{A(n)}$ may be thought of as the cohomology of the generic point.
    \item Suppose $X$ is a smooth variety and let $k(X)$ be its function field. 
   Then, by Lemma \ref{bmtoet}, $\nrcoh{i}{0}{X}{A(n)}$ is the direct limit of 
   the cohomology groups $H^i(U,A(n))$ with $U\subset X$ open and non-empty.   
   Therefore, the group $\nrcoh{i}{0}{X}{A(n)}$ is related to the Galois cohomology of $k(X)$. For example, since \'etale cohomology commutes with filtered limits of quasi-projective varieties with affine transition maps by  \cite[p.\ 88, III.1.16]{milne},
    we have an isomorphism
    $$\nrcoh{i}{0}{X}{\rou{\ell^r}{n}}=\underset{\emptyset \neq U\subseteq X}
    {\varinjlim}\coh{i}{U}{\rou{\ell^r}{n}}\simeq
   \coh{i}{  \varprojlim U}{\rou{\ell^r}{n}}\simeq\coh{i}{k(X)}{\rou{\ell^r}{n}}.$$
   The analogous result fails for integral coefficients, because inverse and direct limits do not commute with each other in general. 
\end{enumerate}
\end{remark}

The Gysin (or localization) sequence \eqref{Gysin} has the following consequence.

\begin{lemma}[\cite{Sch-refined}, Lemma 5.8]\label{exactnesslemma}
Let $X$ be an algebraic scheme. Then, for all integers $i,j$ and $n$,  and any coefficients $A\in \{\Z/\ell^r, \Z_\ell,\Q_\ell\}$, there exists a long exact sequence
$$\dots \rightarrow\nrcoh{i}{j}{X}{A(n)}\rightarrow\nrcoh{i}{j-1}{X}{A(n)}\xrightarrow{\partial}\underset{x\in X^{(j)}}{\bigoplus}\coh{i+1-2j}{x}{A(n-j)}\xrightarrow{i_*}\nrcoh{i+1}{j}{X}{A(n)}\rightarrow \dots .$$
\end{lemma}

Since cohomology vanishes in negative degrees, one deduces the following.

\begin{lemma}[\cite{Sch-refined}, Corollary 5.10]\label{ceilinglemma}
Let $X$ be an algebraic scheme. Then, for all integers $i,j$ and $n$ such that $j\geq\lceil i/2\rceil$, $\nrcoh{i}{j}{X}{A(n)}\simeq\cohbm{i}{X}{A(n)}$, where $A\in\{\Z/{\ell^r} ,\zl,\ql\}$.
\end{lemma}

\section{Galois actions} \label{sec:galois}

\subsection{1-semi-simplicity}
We recall the finite field $\F$ with $q$ elements and its absolute Galois group $G_\F$.
Let $F\in G_\F$ denote the geometric Frobenius morphism, i.e.\ the inverse of the field isomorphism
$$\kbar\rightarrow\kbar,\:x\mapsto x^{q}.$$
We have $G_\F\cong \hat \Z$ and $F$ generates each finite quotient of $G_\F$.  

We denote by $Rep(G_\F, \ql)$ the category of finite-dimensional $\Q_\ell$-vector spaces with a continuous $G_\F$-action.
If $X$ is an algebraic scheme over $\F$, then $H^i(\bar X,\Q_\ell(j))$ and $H^i_{BM}(\bar X,\Q_\ell(j))$  
yield objects in this category.

For an object $V$ of $Rep(G_\F, \ql)$, the sequence
\begin{equation}\label{frob_sequence}
    0\rightarrow V^{G_\F}\rightarrow V\xrightarrow{\identity-F} V\rightarrow V_{G_\F}\rightarrow 0
\end{equation}
is exact.  
To see this, note that $V$ comes from a finitely generated $\zl$-module, say $M$, with a compatible $G_\F$-action. 
It then suffices to prove that $M^{G_\F}=M^{\langle F \rangle}$ and $M_{G_\F}= M_{\langle F \rangle}$.
Both results hold modulo $\ell^r$ for any $r\geq 0$, because any finite quotient of $G_\F$ is generated by $F$.
This implies the result integrally, because $M$ is finitely generated.


\begin{definition}\label{1semisimple_def}
Let $V$ be an object of $Rep(G_\F, \ql)$. We say that ``$V$ is a $1$-semi-simple $G_\F$-representation", or ``$G_\F$ acts $1$-semi-simply on $V$",  or just ``$V$ is $1$-semi-simple" (if the Galois group $G_\F$ is clearly understood from the context)  if the natural morphism $V^{G_\F}\rightarrow V_{G_\F}$ is an isomorphism.
Equivalently, the Frobenius action on $V$ is semi-simple at the eigenvalue $1$.
\end{definition}

The next lemma shows how $1$-semi-simplicity carries over to  extensions.

\begin{lemma}[{\cite[Lemma~12.8]{jannsen-LMN}}]\label{SStrans} 
The following holds:
    \begin{enumerate}
        \item Let $V$ be an object in $Rep(G_\F, \ql)$.
        If $V$ is 1-semi-simple, then any sub-representation and any quotient-representation of $V$ is 1-semi-simple.    
        \item  Let $W\xrightarrow{\phi}V\xrightarrow{\psi}Q$ be an exact sequence of objects of $Rep(G_\F, \ql)$.  If $W$ and $Q$ are $1$-semi-simple $G_\F$ representations and either $V^{G_\F}\rightarrow Q^{G_\F}$ is surjective or $W^{G_\F}=0$ or $Q^{G_\F}$=0, then $V$ is a $1$-semi-simple representation. 
    \end{enumerate}
\end{lemma}

\begin{lemma}\label{SSdef} 
Let $V$ be an object of $Rep(G_\F, \ql)$. Then, $\dim V^{G_\F}=\dim V_{G_\F}$, and the following are equivalent:
\begin{enumerate}
    \item The natural morphism $V^{G_\F}\rightarrow V_{G_\F}$ is injective or surjective;\label{item:SSdef:1}
    \item The natural morphism $V^{G_\F}\rightarrow V_{G_\F}$ is an isomorphism, i.e. $V$ is $1$-semi-simple;\label{item:SSdef:2}
    \item Let $W$ be any $\ql$-representation of $G_\F$ and let $\phi:V\rightarrow W$ be a $G_\F$-equivariant morphism. Let $\omega\in W^{G_\F}$ be the image of some $v\in V$ under $\phi$. Then there exists $v^{\prime}\in V^{G_\F}$ such that $\phi(v^{\prime})=\omega$.\label{item:SSdef:3}
\end{enumerate}
\end{lemma}

\begin{proof}
The claim $\dim V^{G_\F}=\dim V_{G_\F}$ follows from \eqref{frob_sequence} by taking the alternating sum of dimensions.
From this the equivalence of the first two items follows.
To prove that $1$-semi-simplicity implies the property formulated in item \eqref{item:SSdef:3}, let $\phi:V\rightarrow W$ be a $G_\F$-equivariant morphism. 
We may replace $W$ by the image of $\phi$ to assume that $\phi$ is surjective.
The long exact sequence in Galois cohomology then yields an exact sequence
$$
V^{G_\F}\longrightarrow W^{G_\F}\longrightarrow (\ker(\phi))_{G_\F}\longrightarrow V_{G_\F}.
$$
The kernel of $\phi$ is 1-semi-simple because it is a subrepresentation of $V$ (see Lemma \ref{SStrans}).
We can deduce from this that the natural map $(\ker(\phi))_{G_\F}\to V_{G_\F}$ is injective.
This implies surjectivity of $V^{G_\F}\to  W^{G_\F}$ as we want.
Conversely, assume that \eqref{item:SSdef:3} holds true.
Then we can apply that property to the natural surjection $V\to V_{G_\F}$. 
The $G_\F$-action on $V_{G_\F}$ is trivial and so we find that $V^{G_\F}\to V_{G_\F}$ is surjective.
Hence, \eqref{item:SSdef:1} holds and so $V$ is $1$-semi-simple.
\end{proof}

We finish this section by noting that the $1$-semi-simplicity property holds for the image of the cycle class map.

\begin{lemma}\label{SScycle}
Let $X$ be a smooth quasi-projective variety over $\F$. 
Then the absolute Galois group $G_\F$ acts $1$-semi-simply on the image of the cycle class map
$$\cl_{\bar X}^{i} :\chow{i}{\bar X}_{\ql}\rightarrow \coh{2i}{\Xbar}{\ql(i)}.$$
\end{lemma} 
\begin{proof}
The Galois group $G_\F$ acts via finite orbits on cycles and hence on $\chow{i}{\bar X}_{\ql}$.
It follows that there is a finite-dimensional $G_\F$-invariant subspace $V\subset \chow{i}{\bar X}_{\ql}$ that surjects onto $\im(\cl_{\bar X}^{i} )\subset \coh{2i}{\Xbar}{\ql(i)}$.
The Frobenius action on $V$ satisfies $F^m=\id$ for some $m\gg 0$ and so $F$ acts semi-simply on $V$.
Hence it acts semi-simply and in particular 1-semi-simply on the image of $V$,  see  Lemma \ref{SStrans}. 
\end{proof}

\subsection{Galois actions on cohomology}\label{Galois_action_basics}
Let  $G_{\F}$ be the absolute Galois group of the finite field $\F$.  
Let $X$ be an algebraic scheme over $\F$ and let $\Xbar$ be its base change to $\kbar$. 
For $\phi\in G_\F$, we also denote by $\phi$ the (scheme) automorphism of $\Xbar$ given by $\identity \times \phi:X\times_{\F}\kbar\rightarrow X\times_{\F}\kbar$.
We emphasize that this is an automorphism of schemes,  but not an automorphism of $\bar \F$-varieties.  
We have natural isomorphisms of sheaves (on the \'etale or pro-\'etale site) $\phi^\ast \mu_{\ell^r}^{\otimes n}= \mu_{\ell^r}^{\otimes n}$ and $\phi^\ast \pi_X^!\mu_{\ell^r}(n)=\pi_X^!\phi^\ast \mu_{\ell^r}(n)=\pi_X^!\mu_{\ell^r}(n)$.
This implies that there are natural actions
$$
\phi^\ast:H^i(\bar X,\Q_\ell(n))\longrightarrow H^i(\bar X,\Q_\ell(n))\ \ \ \text{and}\ \ \ \phi^\ast:H^i_{BM}(\bar X,\Q_\ell(n))\longrightarrow H_{BM}^i(\bar X,\Q_\ell(n)) ,
$$
which via the  isomorphism of Lemma \ref{bmtoet} agree if $X$ is smooth and equi-dimensional. 

If $\bar U\subset \bar X$ is an open subset whose complement $\bar Z$ has codimension $j$, then the Galois orbit of $\bar Z$ has codimension $j$ as well and so there is an open subset $\bar V\subset \bar U$ that is defined over $\F$ and such that $\bar X\setminus \bar V$ has codimension $j$.
From this observation it follows that 
$$
H^i(F_j\bar X,A(n))\cong \lim_{\substack{\longrightarrow \\ F_j X\subset U\subset  X}} H_{BM}^i(\bar U,A(n)),
$$
where the limit runs over all open subsets $U$ of $X$ that contain $F_jX$.
It is clear from the above description that $H^i(F_j\bar X,A(n))$ admits a natural $G_\F$-action as well.
This action  is functorial and hence restricts to a natural $G_\F$-action on $H^i_{j,nr}(\bar X,A(n))$.
The natural maps
$$
H^i(F_j  X,A(n))\to H^i(F_j\bar X,A(n))\ \ \ \text{and}\ \ \ H^i_{j,nr}(  X,A(n))\to H^i_{j,nr}(\bar X,A(n))
$$
are Galois-equivariant and hence have their images contained in the subspaces of $G_\F$-invariant classes.  

Let us now assume  that $X$ is an equi-dimensional smooth quasi-projective $\F$-scheme. In this case, the Borel-Moore cohomology groups coincide with the $\ell$-adic \'etale cohomology groups by Lemma \ref{bmtoet} and the Galois action on both sides are compatible as explained above.
Since the direct limit is an exact functor,   the Hochschild-Serre short exact sequence in \eqref{HSexact} 
gives the short exact sequence
\begin{equation}\label{HSnrexact}
    0\rightarrow\nrcoh{i-1}{j}{\Xbar}{A(n)}_{G_\F}\rightarrow\nrcoh{i}{j}{X}{A(n)}\rightarrow\nrcoh{i}{j}{\Xbar}{A(n)}^{G_\F}\rightarrow 0
\end{equation}
for every equi-dimensional smooth quasi-projective scheme $X$, for all $i$, $j$, $n$ and for $A \in\{\Z/\ell^r,\zl,\ql\}$.
Here we have used that direct limit commutes with invariance and coinvariance as it is an exact functor.

\begin{lemma}\label{F0pullback}
Let $f:X\to Y$ be a generically finite dominant morphism of quasi-projective $\F$-varieties. 
Then, for all integers $i$ and $n$, there is a $G_\F$-equivariant injection
$$f^*:\nrcoh{i}{0}{\bar{Y}}{\ql(n}\rightarrow\nrcoh{i}{0}{\Xbar}{\ql(n)}.$$ 
\end{lemma}

\begin{proof}
The $G_\F$-equivariance of $f^\ast$ is clear from the construction.
Moreover, we have $f_\ast\circ f^\ast=\deg(f)\id$ on $\nrcoh{i}{0}{\overline{Y}}{\ql(n}$ (see \eqref{eq:f_*f^*}) and so the map in question is injective.
This proves the lemma.
\end{proof}

\begin{corollary}\label{cor:F0vanish}
Let $i,j,d$ be integers and suppose that for every smooth projective $\F$-variety $X$ of dimension $d$, we have
$$
\nrcoh{i}{0}{\bar X}{\ql(j)}^{G_\F}=0.
$$
Then the same holds for all quasi-projective varieties $X$ over $\F$.
\end{corollary}

\begin{proof}
Since $F_0X$ depends only on the generic point of $X$, it suffices to prove that $\nrcoh{i}{0}{\bar X}{\ql(j)}^{G_\F}=0$ for all projective $\F$-varieties $X$.
This reduces to the case of smooth projective $\F$-varieties by Lemma \ref{F0pullback} and the existence of alterations \cite{deJong}. 
\end{proof} 

\subsection{Weight arguments}
Let $X$ be an algebraic scheme over the finite field $\F$ and let $\Xbar$ be the base change of $X$ to an algebraic closure $\kbar$ of $\F$. The geometric Frobenius morphism $F\in G_\F$ acts on the Borel-Moore cohomology groups via $F^\ast$ as explained in Section \ref{Galois_action_basics}.

\begin{definition}\label{weight} 
We say that $V\in Rep(G_\F,\Q_\ell)$  has weights $\geq m$ (resp.\ $\leq m$) if all eigenvalues of the geometric Frobenius action on $V$ have via any embedding $\bar \Q_\ell\to \C$ absolute value $\geq q^{m/2}$ (resp.\ $\leq q^{m/2}$).  
\end{definition}

Deligne's theory of weights (see \cite{deligne-weil, deligne-weil2})  implies:

\begin{theorem}[Deligne]\label{smoothweight}
Let $X$ be a smooth quasi-projective variety over $\F$. 
Then the weights of $\coh{i}{\Xbar}{\ql(j)}$ are greater than or equal to $i-2j$.
\end{theorem}

We note the following well-known implication of Deligne's result, cf.\ \cite{CTSS}.

\begin{lemma}\label{coh_injection}
    Let $X$ be a smooth projective variety over $\F$. 
    Then, for all $i$ and $j$ with $i-2j-1<0$, the natural map
 $\coh{i}{X}{\ql(j)}\rightarrow\coh{i}{\Xbar}{\ql(j)}$ 
 is injective.
    \end{lemma}

\begin{proof}
By the exactness of sequence (\ref{HSexact}), the kernel of the morphism
    $$\coh{i}{X}{\ql(j)}\rightarrow\coh{i}{\bar X}{\ql(j)}$$
    is $\coh{i-1}{\bar X}{\ql(j)}_{G_\F}$. By the Weil conjectures, the geometric Frobenius morphism acting on $\coh{i-1}{\bar X}{\ql(j)}$ has weights $(i-1)-2j<0$.
    In particular, $1$ cannot be an eigenvalue of the geometric Frobenius and so  $\coh{i-1}{\bar X}{\ql(j)}_{G_\F}=0$.
\end{proof}

If 
$$
V=\lim_{\substack{\longrightarrow \\ i}}V_i
$$
 is a direct limit of an inductive system (with compatible $G_\F$-actions) of objects $V_i$ in $Rep(G_\F,\Q_\ell)$, then  we say that $V$ has weights $\geq m$ (resp.\ $\leq m$) if the geometric Frobenius action on $V$ has only eigenvalues of absolute value $\geq q^{m/2}$, resp.\ $\leq q^{m/2}$.
Equivalently,  if for each $i$ the image of $V_i$ in $V$ has  weights $\geq m$ (resp.\ $\leq m$).

\begin{corollary}\label{F0weight}
Let $X$ be a quasi-projective variety over $\F$.
Then  the weights of $\nrcoh{i}{0}{\bar X}{\ql(j)}$ are greater than or equal to $i-2j$.
\end{corollary}

\begin{proof}
Up to shrinking $X$ (which does not change $F_0X$), we can assume that $X$ is smooth.
Then any open subset of $X$ is smooth and so the claim follows from  Theorem \ref{smoothweight}.
\end{proof}

\begin{lemma}\label{BMweight}
Let $X$ be a quasi-projective scheme over $\F$.
Then the weights of the Borel--Moore cohomology group $\cohbm{i}{\Xbar}{\ql(j)}$ are greater than or equal to $i-2j$.
\end{lemma}

\begin{proof}
We may replace $X$ by its reduction to assume that it is generically reduced.
Moreover, the case where $X$ has dimension zero is trivial and so we may argue via induction on $\dim X$.
Let $Z\subset X$ be the singular locus with complement $U$.
The sequence
\begin{equation}\label{Bmweight_seq1}
    \cohbm{i-2c}{\bar{Z}}{\ql(j-c)} \xlongrightarrow{i_*}\cohbm{i}{\bar X}{\ql(j)} \longrightarrow\cohbm{i}{\bar U}{\ql(j)} 
\end{equation}
is exact by \eqref{Gysin}.
The weights of $\cohbm{i}{\bar U}{\ql(j)} $ are $\geq i-2j$ by Theorem  \ref{smoothweight}.
The same holds for  $ \cohbm{i-2c}{\bar{Z}}{\ql(j-c)} $ by induction hypothesis because $\dim Z\leq \dim X$.
This implies the lemma, as we want. 
\end{proof}

\section{Cycle conjectures} \label{sec:cycle-conj}
\subsection{Classical versions}
Let $X$ be a smooth projective variety over the finite field $\F$.
There is a cycle class map
\begin{align} \label{def:cl_X}
\cl_X^i:\CH^i(X)_{\Q_\ell}\longrightarrow H^{2i}(\bar X,i)= H^{2i}(\bar X,\Q_\ell(i))
\end{align}
whose image lands in the subspace of $G_\F$-invariant classes, because the Galois action is trivial on the first factor of $X\times_\F \bar \F$.

\begin{conjecture}[Tate conjecture]\label{Tateconj}  $\im(\cl_X ^i)=H^{2i}(\bar X,\Q_\ell(i))^{G_\F} $.
\end{conjecture}

\begin{conjecture}[Beilinson conjecture]\label{Beilconj}
The cycle class map in \eqref{def:cl_X} 
is injective.
\end{conjecture}

The next conjecture which is often attributed to Grothendieck--Serre (see \cite[Conjecture 12.5]{jannsen-LMN}) concerns the action of $G_\F$ on the cohomology of $\Xbar$.

\begin{conjecture}[Grothendieck--Serre semi-simplicity conjecture]\label{SSconj-strong}
The $G_\F$-action on $H^i(\bar X,\Q_\ell(n))$
 is semi-simple, i.e.\ the Frobenius action is given by a diagonalizable matrix.  
\end{conjecture}

The following is an important weak version of the above conjecture.

\begin{conjecture}[$1$-semi-simplicity conjecture]\label{SSconj}
The $G_\F$-action on $H^i(\bar X,\Q_\ell(n))$
 is $1$-semi-simple, i.e.\ the Frobenius action is given by matrix whose Jordan blocks at the eigenvalue $1$ are trivial. 
 Equivalently, the natural map 
$$\coh{2i}{\Xbar}{\ql(i)}^{G_\F}\rightarrow\coh{2i}{\Xbar}{\ql(i)}_{G_\F}$$
is an isomorphism.
\end{conjecture}

\begin{remark} \label{rem:1-ss_implies_ss}
Clearly, Conjecture \ref{SSconj-strong} implies \ref{SSconj}.
Milne observed that in fact both conjectures are equivalent in the sense that Conjecture \ref{SSconj} for $X\times X$ implies Conjecture \ref{SSconj-strong} for $X$,  see \cite[Remark 8.6]{milne-AJM}.
We recall the argument for convenience.
Let $E_\lambda \subset H^i(\bar X,\bar \Q_\ell(n)) $ be the generalized eigenspace  for the Frobenius action with eigenvalue $\lambda\in \bar \Q_\ell$.
If $E _\lambda\neq 0$, then there is a non-zero class $v\in H^{2d-i}(\bar X,\bar \Q_\ell(d-n))$ which is a Frobenius-eigenvector for the eigenvalue $\lambda^{-1}$,  where $d=\dim X$.
But then 
$$
E _\lambda\otimes v\cdot \bar \Q_\ell \subset H^{2d}(\bar X\times \bar X,\bar \Q_\ell(d))
$$ 
lies in the generalized  eigenspace for eigenvalue $1$ 
and the $1$-semi-simplicity for $X\times X$ implies that the Frobenius action on $E_\lambda$ must be semi-simple,  hence $E_\lambda$ is in fact an eigenspace for the eigenvalue $\lambda$, as we want.
\end{remark}

We also have the following result due to Milne, see \cite[Proposition~8.2 and Remark~8.5]{milne-AJM}.

\begin{proposition}[Milne]\label{prop:Remark_after_conjectures} 
The Tate conjecture for divisors on $X$ implies the $1$-semi-simplicity conjecture for $H^2(\bar X,\mathbb Q_\ell(1))$.  
\end{proposition}


%
\begin{remark}\label{Remark_after_conjectures} 
\begin{enumerate} 
    \item\label{Remark_after_conjectures_2} It follows from Poincar\'e duality that the $1$-semi-simplicity conjecture in degree $i$ and dimension $d$ is equivalent to the $1$-semi-simplicity conjecture in degree $d-i$ and dimension $d$.  
    \item  It follows from \cite[Proposition 8.4]{milne-AJM} (see also \cite[Theorem~2]{moonen}) that the Tate conjecture and the equality $\sim_{num}=\sim_{hom}$, i.e. the coincidence of the numerical and homological equivalence of algebraic cycles with rational coefficients, for all varieties over $\F$ implies the semi-simplicity conjecture over $\F$.
    \item\label{Remark_after_conjectures_5} The Beilinson conjecture holds for divisors, cf.\ \cite[Remark~6.15(a)]{jannsen}. It also holds for zero-cycles, because it can be reduced to the case of divisors using the fact that every zero-cycle is supported on some smooth ample curve $C$ and $\Pic^0(\bar C)$, the Picard group of degree $0$ divisors over $\bar \F$, is a torsion group because every element is defined over a finite field, hence contained in a finite group by Grothendieck's representability of the Picard functor.
    \end{enumerate}
\end{remark}

Conjectures \ref{Tateconj}, \ref{Beilconj} and \ref{SSconj}  are stated for cycles in every degree. To be able to talk about these conjectures only in certain degrees we introduce the following notation:
\begin{notation}\label{deftate}
Let $d$ and $i$ be non-negative integers and recall the finite field $\F$ that is fixed throughout.
\begin{itemize}
    \item $\bigtate{i}{d}{\F}\coloneqq$ The Tate conjecture (Conjecture \ref{Tateconj}) in degree $i$ for every $d$-dimensional smooth projective $\F$-variety,
    \item $\bigSS{i}{d}{\F}\coloneqq$ The $1$-semi-simplicity conjecture (Conjecture \ref{SSconj}) in degree $i$ for every $d$-dimensional smooth projective $\F$-variety,
    \item $\bigbeil{i}{d}{\F}\coloneqq$ The Beilinson conjecture (Conjecture \ref{Beilconj}) in degree $i$ for every $d$-dimensional smooth projective $\F$-variety.
\end{itemize}
\end{notation}
We finish this section with a simple lemma.

\begin{lemma}\label{Tatedown}
Let $i$ and $d$ be non-negative integers. If $\bigtate{i}{d}{\F}$ (respectively $\bigSS{i}{d}{\F}$, $\bigbeil{i}{d}{\F}$) holds, then for all non-negative integers $p$ and $q$ satisfying $0\leq q-p\leq d-i$, $\bigtate{i-p}{d-q}{\F}$ (respectively $\bigSS{i-p}{d-q}{\F}$, $\bigbeil{i-p}{d-q}{\F}$) holds as well.
\begin{proof}
Let $X$ be a $(d-1)$-dimensional smooth projective variety over $\F$. 
We have direct sum decompositions 
$$\chow{i}{X\times\mathbb{P}^1}\simeq\chow{i}{X}\oplus\chow{i-1}{X}$$
and
$$\coh{2i}{\Xbar\times\mathbb{P}^1}{\ql(i)}\simeq\coh{2i}{\Xbar}{\ql(i)}\oplus\coh{2i-2}{\Xbar}{\ql(i-1)},$$
which are compatible with the cycle class map and the action of the Galois group. From this it follows that $\bigtate{i}{d}{\F}$ implies $\bigtate{i}{d-1}{\F}$ and $\bigtate{i-1}{d-1}{\F}$ and similarly for $\bigSS{i}{d}{\F}$ and $\bigbeil{i}{d}{\F}$. Using this inductively gives the result.
\end{proof}
\end{lemma}

\subsection{Generalizations to arbitary varieties} \label{subsec:Tateetc-open}

Jannsen has formulated analogues of the Hodge and Tate  conjecture for arbitrary varieties (not necessarily smooth nor proper), see \cite[Conjecture 7.2 and 7.3]{jannsen-LMN}. 
He has shown that these variants essentially reduce to the case of smooth proper varieties, see \cite[Theorems 7.9 and 7.10]{jannsen-LMN},  if one assumes the semi-simplicity conjecture. 
(Jannsen also assumes resolution of singularities, but after the work of de Jong this can be replaced by alterations \cite{deJong}.)

 Propositions \ref{SSopenvar},  \ref{prop:Tateopenvar} and   \ref{SSsingular} below show that the Tate and $1$-semi-simplicity conjectures for smooth projective varieties imply their analogues for open or singular varieties under certain conditions. They are similar to \cite[Theorem~7.10]{jannsen-LMN} and \cite[Theorem~12.7]{jannsen-LMN}. 
In all cases, Jannsen proves stronger statements under stronger assumptions, but his ideas can be adapted to prove the versions stated below. 

\begin{proposition}\label{SSopenvar} 
 If 
$\coh{2i}{\Xbar}{\ql(i)} $ is 1-semi-simple
  for every $d$-dimensional smooth projective $\F$-variety $X$, then the same is true for every $d$-dimensional smooth quasi-projective $\F$-variety.
\end{proposition}

\begin{proof}
Let $U$ be a smooth quasi-projective $\F$-variety of dimension $d$.
By \cite{deJong}, applied to a projective closure of $U$, we get a smooth projective variety $X$ with a dense open subset $V\subset X$ together with a generically finite proper surjective map $\tau : V\to U$. 
By \eqref{eq:f_*f^*}, we have
$$
 \tau_*\circ \tau^*=\deg\tau\cdot \identity :\coh{2i}{\overline{U}}{\ql(i)}\longrightarrow\coh{2i}{\overline{U}}{\ql(i)} .
$$
Hence,  the pullback map $\tau^\ast$ is injective and so it suffices by Lemma  \ref{SStrans} to prove that $H^{2i}(\bar V,\Q_\ell(i))$ is 1-semi-simple.
In other words, we may assume that $U=V$ admits a smooth projective closure $X$.
Let $Z=X\setminus U$.
Then the sequence
\begin{equation}\label{SSopen_seq1}
    \coh{2i}{\bar{X}}{\ql(i)}\rightarrow\coh{2i}{\bar{U}}{\ql(i)}\rightarrow\cohbm{2i-2c+1}{\bar Z}{\ql(i-c)}
\end{equation}
is exact by the exactness of (\ref{Gysin}), where we used Lemma \ref{bmtoet} to identify the Borel-Moore cohomology of $\bar {U}$ and $\bar X$ with the $\ell$-adic \'etale cohomology of $\bar{U}$ and $\bar X$, respectively. 

Since $X$ is smooth and projective,  $G_\F$ acts $1$-semi-simply on $\coh{2i}{\bar{X}}{\ql(i)}$ by assumptions. 
On the other hand, $\cohbm{2i-2c+1}{\bar{Z}}{\ql(i-c)}$ has weights $\geq 1$ (see Lemma \ref{BMweight}) and so the Frobenius action does not have $1$ as an eigenvalue, which shows that $1$-semi-simplicity holds for trivial reasons. 
We then conclude from Lemma \ref{SStrans} that $\coh{2i}{\bar{U}}{\ql(i)}$ is a $1$-semi-simple representation of $G_\F$, as we want.
\end{proof}

\begin{proposition}\label{prop:Tateopenvar} 
If  for every $d$-dimensional smooth projective  variety $X$ over the finite field $\F$, 
the cycle class map
$$cl^{i}_{X}:\chow{i}{X}_{\ql}\longrightarrow\coh{2i}{\bar X}{\ql(i)}^{G_\F}$$
is surjective and 
 $\coh{2i}{\Xbar}{\ql(i)}$ is 1-semi-simple,
then the same is true for every $d$-dimensional smooth quasi-projective $\F$-variety.
\end{proposition}

\begin{proof}
By Proposition \ref{SSopenvar}, we already know the $1$-semi-simplicity part of the statement. So it remains to prove the surjectivity of the cycle class map.

Let $U$ be a smooth quasi-projective variety over $\F$.
Using alterations \cite{deJong}, one reduces the problem to the case where $U$ admits a smooth projective closure $X$.
Let $Z:=X\setminus U$.
Then there is an exact sequence \eqref{SSopen_seq1}, where $\cohbm{2i-2c+1}{\bar Z}{\ql(i-c)}$ has no Galois invariants by Lemma \ref{BMweight}.
By $1$-semi-simplicity of  $\coh{2i}{\bar{X}}{\ql(i)}$, the natural map
$$
\coh{2i}{\bar{X}}{\ql(i)}^{G_\F}\longrightarrow \coh{2i}{\bar{U}}{\ql(i)}^{G_\F}
$$
is surjective by  Lemma \ref{SSdef}.
The surjectivitiy of $\cl_X^i$ then implies the surjectivity of $\cl_U^i$, as we want. 
\end{proof}

\begin{proposition}\label{SSsingular}
Let $X$ be a quasi-projective scheme over the finite field $\F$ and assume that $\bigtate{j}{d}{\F}$ and $\bigSS{j}{d}{\F}$ hold.  
Then $\cohbm{2i}{\bar X}{\ql(i)}$ is 1-semi-simple for all 
 $i$ with $\dim X-d+j\leq i\leq j$.
\end{proposition}
\begin{proof}
We will do an induction on the dimension of $X$. Thanks to Lemma \ref{directsum}, we may assume throughout the proof that all schemes are reduced.

Let $U:=X_{reg}$ denote the regular locus of $X$ and let $Z\coloneqq X-U$ denote its complement. 
Let $c\coloneqq\dim X-\dim U$  and let $i$ be an integer satisfying the inequalities $\dim X-d+j\leq i\leq j$. 
The sequence
\begin{equation}\label{gysinSSsing}
    \cohbm{2i-2c}{\bar{Z}}{\ql(i-c)}\longrightarrow\cohbm{2i}{\bar{X}}{\ql(i)}\longrightarrow\cohbm{2i}{\bar U}{\ql(i)}
\end{equation}
is exact by the exactness of (\ref{Gysin}).
By  Lemmas \ref{directsum} and   \ref{bmtoet} we have
\begin{align}\label{eq:oplusU_s}
\cohbm{2i}{\bar U}{\ql(i)}\cong \bigoplus_s \coh{2i-2c_s}{\bar U_s}{\ql(i-c_s)}
\end{align}
where $U_s$ denote the irreducible components of $U$ and where $c_s=\dim U-\dim U_s$.

We claim that for each component $U_s$, $G_\F$ acts $1$-semi-simply on $\coh{2i-2c_s}{\bar U_s}{\ql(i-c_s)}$ and the morphism
$$
\chow{i-c_s}{U_s}_{\ql}\longrightarrow \coh{2i-2c_s}{\bar U_s}{\ql(i-c_s)}^{G_\F}
$$
is surjective. 
To see this, we apply Proposition \ref{prop:Tateopenvar} to reduce the problem to showing that $\bigtate{i-c_s}{\dim X-c_s}{\F}$ and $\bigSS{i-c_s}{\dim X-c_s}{\F}$ hold.
By Lemma \ref{Tatedown}, we reduce this further to the validity of $\bigtate{j}{d}{\F}$ and $\bigSS{j}{d}{\F}$, which hold by assumptions. 
This last step works because we can rewrite $i-c_s$ and $\dim X-c_s$ as $i-c_s=j-(j-i+c_s)$ and $\dim X-c_s=d-(d-\dim X+c_s)$, and then apply Lemma \ref{Tatedown} because $(j-i+c_s)$, $(d-\dim X+c_s)$ are non-negative and $j-i+c_s\leq d-\dim X+c_s$ by the assumption on $i$. 
In conclusion we have thus shown that the Galois module \eqref{eq:oplusU_s} is $1$-semi-simple and the subspace of $G_\F$-invariant classes is algebraic.

 Next, $\cohbm{2i-2c}{\bar {Z}}{\ql(i-c)}$ is 1-semi-simple by the induction hypothesis because  $\dim Z<\dim X$
 and
 $\dim Z-d+j=\dim X-c-d+j\leq i-c\leq j $.
 In order to prove 1-semi-simplicity for $\cohbm{2i}{\bar{X}}{\ql(i)}$ it now suffices by  Lemma \ref{SStrans} to prove that
$$
\cohbm{2i}{\bar {X}}{\ql(i)}^{G_\F}\longrightarrow\cohbm{2i}{\bar U}{\ql(i)}^{G_\F}
$$
is surjective.
This in turn follows from the fact that the target of the above map consists of algebraic classes,  as we have shown above.
This concludes the proof of the proposition.
\end{proof}

\subsection{Generalizations to higher Chow groups}

For a smooth $\F$-scheme $X$ and a ring $R$,  we denote by 
$$
H^p_M(X,R(q)):=\mathbb H^p(X_{Zar},(z^n(-,\bullet)\otimes_{\Z} R)[-2n])
$$
 the motivic cohomology of $X$ with coefficients in $R$, 
 where $ z^n(-,\bullet)$ denotes the complex of Zariski sheaves given by Bloch's cycle complex \cite{bloch-higher-chow}. 
There are natural isomorphisms
$$
H^p_M(X,\Z(q))\cong
\CH^q(X,2q-p) 
$$
to the higher Chow groups of Bloch, see \cite[p.\ 269, (iv)]{bloch-higher-chow}.
In particular,  $H^{2p}_M(X,\Z(p))\cong \CH^p(X)$ agrees with ordinary Chow groups.

There is a canonical map $H^p_M(X,\Z(q))\to H^p_L(X,\Z(q))$ to Lichtenbaum motivic cohomology, where the latter is given by the hypercohomology of the aforementioned motivic complex in the \'etale topology:
$$
H^p_L(X,R(q)):=\mathbb H^p(X_{\et},(z^n(-,\bullet)\otimes_{\Z} R)[-2n]) .
$$
There is a reduction modulo $\ell^r$-map $H^p_L(X,\Z(q))\to H^p_L(X,\Z/\ell^r(q))$ and we have a canonical isomorphism 
$$
H^p_L(X,\Z/{\ell^r}(q))\cong H^p(X,\mu_{\ell^r}^{\otimes n}),
$$
where the right hand side denotes ordinary \'etale cohomology, see \cite[Theorem 1.5]{geisser-levine}.
Taking inverse limits, we thus arrive at a natural map
$$
H^p_M(X,\Z(q))\otimes_{\Z} \Z_\ell\longrightarrow H^p(X,\Z_\ell(n)).
$$
We can compose this with the natural map $H^p(X,\Z_\ell(n))\to H^p(\bar X,\Z_\ell(n))  $ to get a higher cycle map to $H^p(\bar X,\Z_\ell(n))$, which lands  in the subspace of Galois-invariant classes.

\begin{conjecture}[Tate conjecture for higher Chow groups] \label{conj:generalized-Tate-function}
Let $X$ be a smooth quasi-projective scheme over the finite field $\F$. 
Then the natural map 
$$
H^p_M(X,\Z(q))\otimes_{\Z } \Q_\ell\longrightarrow H^p(\bar X,\Q_\ell(q))^{G_\F}
$$
is surjective.
\end{conjecture}

\begin{conjecture}[Beilinson conjecture for higher Chow groups]\label{conj:generalized-Beilinson}
Let $X$ be a smooth quasi-projective variety over the finite field $\F$. 
Then the natural map 
$$
H^p_M(X,\Z(q))\otimes_{\Z}  \Q_\ell\longrightarrow H^p(\bar X,\Q_\ell(q))^{G_\F}
$$
is injective.
\end{conjecture}

The conjunction of the Tate and Beilinson conjecture for higher Chow groups says that $H^p_M(X,\Z(q))\otimes_{\Z}  \Q_\ell\cong H^p(\bar X,\Q_\ell(q))^{G_\F}$; this goes back to Friedlander and Beilinson, see \cite[Remark 8.3.4(b)]{beilinson-absolute-hodge} and  \cite[Conjecture 12.4(b)]{jannsen-LMN}.

\begin{remark}
The Tate conjecture for higher Chow groups applied to smooth projective varieties is equivalent to the ordinary Tate conjecture, because $H^p(\bar X,\Q_\ell(q))^{G_\F}=0$ for $p\neq 2q$ and smooth projective $X$ for weight reasons.
We have shown in Proposition \ref{prop:Tateopenvar} above that the Tate conjecture and the 1-semi-simplicity conjecture for smooth projective varieties implies the respective statements for all smooth quasi-projective varieties. 
In contrast, we do not know if Conjecture \ref{conj:generalized-Tate-function} for smooth projective varieties (which by the previous comment is equivalent to the ordinary Tate conjecture) together with the semi-simplicity conjecture  implies Conjecture \ref{conj:generalized-Tate-function} for all smooth quasi-projective varieties.
In fact, if one could show this, then the main result of this paper (see Theorem \ref{thm:main:Tate-Beilinson-ss}) together with Lemma \ref{lem:generalized-Tate} below would show that the Tate conjecture (Conjecture \ref{Tateconj}) and 1-semi-simplicty conjecture implies the Beilinson conjecture (Conjecture \ref{Beilconj}).
\end{remark}

For a smooth $\F$-scheme $X$, we define
$$
H^p_M(F_jX,\Z(q)):=\lim_{\substack{\longrightarrow\\F_jX\subset U\subset X}}
H^p_M(U,\Z(q)).
$$
Motivic cohomology commutes with projective systems of schemes with affine dominant transition maps,  see e.g.\ \cite[Example 11.1.25]{cisinski-deglise}. 
Hence, for any smooth quasi-projective variety $X$ over $\F$, we have
\begin{align} \label{eq:motivic coho of F(X)}
H^p_M(F_0X,\Z(q)) \cong H^p_M(\Spec \F(X),\Z(q))\cong \CH^{q}(\Spec \F(X),2q-p).
\end{align}
A class in $\CH^{q}(\Spec \F(X),2q-p)$ is represented by a codimension $q$-cycle on the algebraic $(2q-p)$-simplex over $\F(X)$.
If $2q-p<q$, or equivalently $q<p$, then no such cycle exists by dimension reasons and so
\begin{align} \label{eq:motivic coho of F(X)-vanishing}
H^p_M(F_0X,\Z(q))=0\ \ \ \text{for $q<p$}. 
\end{align}  
Conjecture \ref{conj:generalized-Tate-function} thus implies the following statement for function fields.

\begin{conjecture}[Tate conjecture for function fields] \label{conj:generalized-Tate-function field}
Let $X$ be a smooth quasi-projective variety over the finite field $\F$. 
Then the natural map 
$$
H^p_M(\Spec  \F(X),\Z(q)) \otimes_{\Z } \Q_\ell\longrightarrow H^p(F_0 \bar X,\Q_\ell(q))^{G_\F}
$$
is surjective.
In particular,
$H^p(F_0 \bar X,\Q_\ell(q))^{G_\F}=0$ for $q<p$.
\end{conjecture}

\begin{lemma}\label{lem:generalized-Tate}
Conjecture \ref{conj:generalized-Tate-function} implies Conjecture \ref{conj:generalized-Tate-function field}.
\end{lemma}
\begin{proof}
The surjectivity statement in Conjecture \ref{conj:generalized-Tate-function field} follows from \eqref{eq:motivic coho of F(X)} and the fact that the direct limit functor is exact.
The remaining assertion follows from \eqref{eq:motivic coho of F(X)-vanishing}.
\end{proof}

\begin{remark}\label{rem:beilinson-conjecture} 
By a conjecture of Beilinson and Parshin, 
$$
H^p_M(\Spec \bar \F(X),\Z(q)) \otimes_{\Z } \Q_\ell=0\ \ \ \text{ for all $p\neq q$},
$$
see \cite[Conjecture 8.3.3(a)]{beilinson-absolute-hodge}.
(Beilinson considers in loc.\ cit.\ a different definition for motivic cohomology, which however coincides rationally with the one used in this paper, see \cite[p.\ 269, (iv) and (vii)]{bloch-higher-chow}.
Moreover, Beilinson attributes an equivalent conjecture on the algebraic K-theory of smooth varieties over $\F$ to Parshin in \cite[Conjecture 2.4.2.3]{beilinson-2}.)
Conjecture \ref{conj:generalized-Tate-function} thus implies 
$$
H^p(F_0 \bar X,\Q_\ell(q))^{G_\F}=0\ \ \ \text{ for $p\neq q$}.
$$
For $p=q$, we have that $H^p_M(\Spec  \F(X),\Z(p))\cong K^M_p(\F(X))$ agrees with the $p$-th Milnor K-group of $ \F(X)$, see \cite[Theorem 5.1 and 19.1]{mazza-voevodsky-weibel}.
Hence, the aforementioned conjecture of Beilinson together with Conjecture \ref{conj:generalized-Tate-function} and Conjecture \ref{conj:generalized-Beilinson} for $p=q$ implies that the invariant $H^p(F_0 \bar X,\Q_\ell(q))^{G_\F}$ of the function field $\F(X)$  that we study in this paper should agree with Milnor K-theory tensored with $\Q_\ell$, as follows
$$
H^p(F_0 \bar X,\Q_\ell(q))^{G_\F}\cong \begin{cases} K^M_p(\F(X))\otimes_{\Z} \Q_\ell\ \ \ \ &\text{if $p=q$};\\
0\ \ \ \ &\text{otherwise}.
\end{cases}
$$
\end{remark}

\section{More on 1-semi-simplicity} \label{sec:more-ss}

\begin{lemma}\label{SSopen}
Let $X$ be a quasi-projective variety over the finite field $\F$ and let $i\geq 0$.
If the absolute Galois group $G_\F$ acts $1$-semi-simply on $\cohbm{2i}{\bar X}{\ql(i)}$, then the same is true for all open subsets $U\subseteq X$.
\end{lemma}
\begin{proof}
Let $Z=X\setminus U$  and let $c\coloneqq\dim X-\dim Z$. 
The sequence
$$
\cohbm{2i}{\bar X}{\ql(i)}\longrightarrow\cohbm{2i}{\bar{U}}{\ql(i)}\longrightarrow\cohbm{2i+1-2c}{\bar{Z}}{\ql(i-c)}
$$
is exact by the exactness of (\ref{Gysin}).
By Lemma \ref{BMweight},  $\cohbm{2i+1-2c}{\bar{Z}}{\ql(i-c)}$ has weights $\geq 1$, hence no Galois-invariants and so it is 1-semi-simple, see Lemma \ref{SSdef}. 
 On the other hand, $\cohbm{2i}{\bar X}{\ql(i)}$ is $1$-semi-simple by assumptions. The $1$-semi-simplicity of $\cohbm{2i}{\bar{U}}{\ql(i)}$ follows from Lemma \ref{SStrans}.
\end{proof}

\begin{lemma}\label{lem:SSopencor}
Let $X_1,...,X_n$ be quasi-projective varieties over the finite field $\F$ such that the absolute Galois group $G_\F$ acts $1$-semi-simply on $\cohbm{2i}{\bar X_m}{\ql(i)}$ for all $m=1,...,n$. 
Let $V$ be a $\Q_\ell$-vector space (possibly infinite dimensional) with a linear $G_\F$-action.
If 
$$
\phi:\bigoplus_{m=1,\dots ,n}\nrcoh{2i}{j}{\bar X_m}{\ql(i)}\longrightarrow V
$$
is a $G_\F$-equivariant $\Q_\ell$-linear map, then 
any class $v\in \im(\phi)$ that is Galois-invariant admits a Galois-invariant lift.
\end{lemma}

\begin{proof}
Replacing $V$ by $\im(\phi)$, we may assume that $\phi$ is surjective.
Let $v\in V$ be $G_\F$-invariant. 
We may find open subsets $U_m\supseteq F_{j}X_m$ such that $v$ is in the image of the composition
$$\bigoplus\cohbm{2i}{\bar U_m }{\ql(i)}\longrightarrow\bigoplus\nrcoh{2i}{j}{\bar X_m}{\ql(i)}\longrightarrow V.
$$
The left hand side is a finite-dimensional $\Q_\ell$-vector space with a 1-semi-simple Galois action, see Lemma \ref{negdegree} and \ref{SSopen}.
It follows from Lemma \ref{SSdef} that the Galois-invariant element $v\in V$ lifts to a Galois-invariant class via $\phi$, as we want.
\end{proof}

\begin{lemma}\label{lem:liftforbeil2}
Let $X$ be a smooth projective variety over the finite field $\F$. 
Let $U$ be an open subset of $X$ such that $F_{j}X\subseteq U$. 
Assume that $\bigtate{i-j-1}{\dim X-j-1}{\F}$ and $\bigSS{i-j-1}{\dim X-j-1}{\F}$ hold.
Then the absolute Galois group $G_\F$ acts $1$-semi-simply on $\coh{2i-1}{\bar{U}}{\ql(i)}$.\par
The $j=0$ case holds when $X$ is just projective and $U$ is smooth.
\end{lemma}

\begin{proof}
By the Gysin sequence (\ref{Gysin}), there exists an exact sequence
$$
\coh{2i-1}{\bar X}{\ql(i)}\longrightarrow\coh{2i-1}{\bar{U}}{\ql(i)}\longrightarrow\cohbm{2i-2c}{\bar{Z}}{\ql(i-c)},
$$
where $Z\coloneqq X-U$  and $c\coloneqq\dim X-\dim Z$. Notice that $c\geq j+1$ by assumptions.
By the Weil conjectures,  $\coh{2i-1}{\bar X}{\ql(i)}^{G_\F}=0$ and so $\coh{2i-1}{\bar X}{\ql(i)}$ is trivially a $1$-semi-simple $G_\F$-representation, see Lemma \ref{SSdef}. 
By Proposition \ref{SSsingular}, $\bigtate{i-j-1}{\dim X-j-1}{\F}$ and $\bigSS{i-j-1}{\dim X-j-1}{\F}$ imply that $\cohbm{2i-2c}{\bar{Z}}{\ql(i-c)}$ is $1$-semi-simple as well, where we used that
$
i-c\leq i-j-1
$ and 
$$
\dim Z-( \dim X-j-1)+i-j-1= \dim Z-\dim X+i=i-c.
$$
It follows from Lemma \ref{SStrans} that $\coh{2i-1}{\bar{U}}{\ql(i)}$ is $1$-semi-simple, as we want.

In the $j=0$ case with $X$ projective and $U$ smooth, the problem may be reduced to the case where $X$ is smooth, as follows. 
By  \cite{deJong}, there exists an alteration $\tau:X^{\prime}\rightarrow  X$ with $X^{\prime}$ smooth and projective. 
Clearly, $F_{0}X^{\prime}\subseteq\tau^{-1}(U)$ (this is the part that does not work for $j>0$). 
Moreover, since $\tau$ gives an injection 
$$
\coh{2i-1}{\bar{U}}{\ql(i)}\hookrightarrow\coh{2i-1}{\tau^{-1}(\bar{U})}{\ql(i)}
$$
because $\tau_*\tau^*=\deg\tau\cdot \identity $, see \eqref{eq:f_*f^*},
it suffices to show the $1$-semi-simplicity of $\coh{2i-1}{\tau^{-1}(\bar{U})}{\ql(i)}$ by Lemma \ref{SStrans}.
This reduces the claim to the case where $X$ is smooth projective,  proven above.
\end{proof}

\begin{lemma}\label{lem:liftforbeil2cor}
Assume that $\bigtate{i-j-1}{d-j-1}{\F}$ and $\bigSS{i-j-1}{d-j-1}{\F}$ hold for some non-negative integers $i,j,d$.
Let $X_1,...,X_n$ be smooth projective $\F$-varieties of dimension $d$ and let $V$ be a $\Q_\ell$-vector space with a linear $G_\F$-action.
Let
$$
\phi:\bigoplus_{m=1}^n \nrcoh{2i-1}{j}{\bar X_m}{\ql(i)}\longrightarrow V
$$
 be a $G_\F$-equivariant $\Q_\ell$-linear map.
Then any $G_\F$-invariant class in $\im (\phi)$ admits via $\phi$ a $G_\F$-invariant lift.

The $j=0$ case of the statement above holds without assuming smoothness of $X_1,...,X_n$.
\end{lemma}

\begin{proof}
Let $v\in \im (\phi)$ be  Galois-invariant. 
We may find open subsets $U_m\supseteq F_{j}X_m$ of $X_m$ such that $v$ is in the image of the composition
$$
\bigoplus_{m=1}^n \coh{2i-1}{\bar U_m}{\ql(i)}\longrightarrow\bigoplus_{m=1}^n \nrcoh{2i-1}{j}{\bar X_m}{\ql(i)}\longrightarrow V.
$$
By Lemma \ref{lem:liftforbeil2}, the left hand side is $1$-semi-simple and the same holds for $j=0$ without the smoothness assumption on $X_1,\dots ,X_n$, as long as we pick each $U_m\subset X_m$ small enough so that it is smooth.
It thus follows from Lemmas \ref{negdegree} and \ref{SSdef} that $v$ admits a Galois-invariant lift, as we want.
\end{proof}

\begin{cor}\label{cor:liftforbeil}
Let $X_1,...,X_n$ be smooth projective varieties of dimension $d$ over the finite field $\F$. 
 Let $V$ be a $\Q_\ell$-vector space with a linear $G_\F$-action and let 
 $$
 \phi:\bigoplus_{m=1,\dots ,n}\nrcoh{2i-1}{i-1}{\bar X_m}{\ql(i)}\longrightarrow V
 $$ 
 be a $G_\F$-equivariant morphism. 
 If  $v\in \im (\phi)$ is $G_\F$-invariant, then $v$ admits via $\phi$ a $G_\F$-invariant lift.
 
    The $i=1$ case of the statement above holds without assuming smoothness of $X_1,...,X_n$.
\begin{proof}
    This is the $j=i-1$ case of Lemma \ref{lem:liftforbeil2cor}; 
    the assumptions $\bigtate{0}{d-i}{k}$ and $\bigSS{0}{d-i}{k}$ of that lemma are trivially true.
\end{proof}
\end{cor}

\section{Galois-invariants of the cohomology of the generic point} \label{sec:galois-invariants}

The purpose of this section is to prove the following result which relates the validity of the 1-semi-simplicity, the Beilinson, and the Tate conjecture for $X$ with the Galois-invariants of the generic point.
The latter are expected to vanish by the Tate conjecture for the function field of $X$, see Conjecture \ref{conj:generalized-Tate-function field}.

\begin{theorem}\label{thm:induct1} 
Let $X$ be a smooth projective variety of dimension $d$ over the finite field $\F$.
Then the following hold true.
\begin{enumerate}
    \item\label{induct1_1} If the Tate and 1-semi-simplicity conjecture hold in degree $i$ on $X$,  then 
    $$
    \nrcoh{2i}{0}{\bar X}{\ql(i)}^{G_\F}=0.
    $$ 
    \item\label{induct1_2} Assume that $\bigtate{i-1}{d-1}{\F}$ and $\bigSS{i-1}{d-1}{\F}$ hold.
    If
    $$
   \nrcoh{2i}{0}{\bar X}{\ql(i)}^{G_{\F }}=0,
    $$
    then the Tate and 1-semi-simplicity conjecture hold in degree $i$ on $X$. 
    \item\label{induct1_3} For all $d$ and $i\geq 2$,  assume that $\bigtate{i-1}{d-1}{\F}$ and $\bigSS{i-1}{d-1}{\F}$  hold. 
    If the Beilinson conjecture for codimension $i$-cycles on $X$ holds, then 
    $$
    \nrcoh{2i-1}{0}{\bar X}{\ql(i)}^{G_\F}=0 .
    $$ 
    \item\label{induct1_4} 
    Assume that $\bigtate{i-2}{d-2}{\F}$, $\bigSS{i-2}{d-2}{\F}$, and $\bigbeil{i-1}{d-1}{\F}$ hold.
    If
    $$
   \nrcoh{2i-1}{0}{\bar X}{\ql(i)}^{G_{\F }}=0
    $$
    Then the Beilinson conjecture holds in degree $i$ on $X$.
\end{enumerate}
\end{theorem}

\subsection{Refined unramified cohomology and algebraic cycles over finite fields}

Our starting point is the following result.
 
\begin{theorem}[\cite{Sch-refined}]\label{kercyc_1}
Let $X$ be a smooth quasi-projective variety over the finite field $\F$. 
Then,
$$
\ker(\cl^{i}_X:\CH^i(X)_{\Q_\ell}\to H^{2i}(\bar X,\Q_\ell(i)))\simeq\frac{\urcoh{2i-1}{i-2}{X}{\ql(i)}}{\coh{2i-1}{X}{\ql(i)}}.
$$
\end{theorem}
\begin{proof}
Since $X$ is smooth and irreducible, hence equi-dimensional,  Borel--Moore cohomology coincides with ordinary cohomology by Lemma \ref{bmtoet}.
Hence, by \cite[Theorem 1.8(1)]{Sch-refined}, we have
$$
\ker(\cl^{i}_X:\CH^i(X)_{\Q_\ell}\to H^{2i}(X,\Q_\ell(i)))\simeq\frac{\urcoh{2i-1}{i-2}{X}{\ql(i)}}{\coh{2i-1}{X}{\ql(i)}},
$$
where we use that  in loc.\ cit.\ we have $N^{i-1}\CH^i(X)_{\Z_\ell}=0$ because $\F$ is a finite field,  hence finitely generated over its prime field, see \cite[Lemma 7.5(2) and Proposition 6.6]{Sch-refined}.
It remains to prove that $H^{2i}( X,\Q_\ell(i)))\to H^{2i}(\bar X,\Q_\ell(i)))$ is injective, which is well-known: 
By \eqref{HSexact} the kernel of this map is given by the coinvariants of $H^{2i-1}(\bar X,\Q_\ell(i)))$, which vanish for weight reasons, see
 Theorem \ref{smoothweight}. This concludes the proof of the theorem. 
\end{proof}

\begin{prop}\label{roof1}
Let $X$ be a smooth projective variety over the finite field $\F$. 
Then,
\begin{enumerate}
    \item the morphism
    \begin{equation}\label{equation_roof1}
    \urcoh{2i-1}{i-2}{X}{\ql(i)}\rightarrow\urcoh{2i-1}{i-2}{\Xbar}{\ql(i)}^{G_\F}
    \end{equation}
    is an isomorphism.
    \item The morphism
    \begin{equation}\label{equation_roof2}
    \urcoh{2i}{i-1}{X}{\ql(i)}\rightarrow\urcoh{2i}{i-1}{\Xbar}{\ql(i)}^{G_\F}
    \end{equation}
    is injective. If $H^{2i}(\bar X,\ql(i))$ is 1-semi-simple, it is surjective as well.    
\end{enumerate}
\end{prop}

\begin{proof}
For a point $x\in X^{(j)}$ we denote by $\bar x:=\Spec \kappa x\times_{\F} \Spec \bar \F$ its base change to $\bar \F$; this is a finite union of codimension $j$ points of $\bar X$ which form a $G_\F$-orbit under the natural Galois action on $\bar X^{j}$.
Note that we have an equality of sets
$$
\{ \bar x\mid x\in X^{(j)} \}=\bar X^{(j)} .
$$
We thus get a commutative diagram
\begin{align}\label{Diagram_roof1}
\xymatrix{
 \underset{x\in X ^{(i-j)}}{\bigoplus}H^j(x,j)\ar[r]\ar[d] & 
 H^{2i-j}(F_{i-j}X,i) \ar[d]^{\phi}\ar[r]&  H^{2i-j}(F_{i-j-1}X,i)  \ar[r]\ar[d]& \underset{x\in X^{(i-j)}}{\bigoplus} H^{j+1}(x,j)\ar[d] \\
 \underset{x\in X ^{(i-j)} }{\bigoplus} H^j(\bar x,j)  \ar[r]& H^{2i-j}(F_{i-j}\bar X,i) \ar[r]& H^{2i-j}(F_{i-j-1}\bar X,i)   \ar[r] &  \underset{x\in X ^{(i-j)}}{\bigoplus}H^{j+1}(\bar{x},j) ,
}
\end{align} 
whose rows are exact by Lemma \ref{exactnesslemma}, whose vertical arrows come from  (\ref{HSnrexact}), and where all cohomology groups are computed with $\Q_\ell$-coefficients.
We will consider the above diagram for $j\in\{0, 1\}$.
Note that in the $j=0$ case, $i_*$ induces the cycle class map, see Lemma \ref{ceilinglemma} and \cite[Lemma 9.1]{Sch-refined}.

First we prove the surjectivity of \eqref{equation_roof1} and \eqref{equation_roof2}. 
Let $\alpha\in\urcoh{2i-j}{i-1-j}{\bar X}{i}^{G_\F}$, i.e. $\alpha$ is $G_\F$-invariant and lifts to a class $\gamma\in\nrcoh{2i-j}{i-j}{\bar X}{i}$. 
We may assume that $\gamma$ is $G_\F$-invariant: This follows from Corollary \ref{cor:liftforbeil} ($n=1$ case) in the $j=1$ case and from Lemma \ref{SSdef} in the $j=0$ case, where we use that $H^{2i}(\bar X,i)=H^{2i}(F_i\bar X,i)$ is 1-semi-simple in the $j=0$ case by assumption.  
Therefore, by the exactness of \eqref{HSnrexact}, $\gamma$ lies in the image of   $\phi$. 
The surjectivity of \eqref{equation_roof1} and \eqref{equation_roof2} follows now from the commutativity of \eqref{Diagram_roof1}.

\begin{claim}
The morphism $\phi$ of diagram (\ref{Diagram_roof1}) is injective. 
\end{claim}
\begin{proof}
By the exactness of \eqref{HSnrexact},
$$
\ker\phi \simeq\nrcoh{2i-j-1}{i-j}{\Xbar}{i}_{G_{\F}}\simeq\coh{2i-j-1}{\Xbar}{i}_{G_{\F}},
$$ 
where the second isomorphism follows from Lemmas \ref{bmtoet} and \ref{ceilinglemma}, because $i-j\geq \lceil \frac{2i-j-1}{2}\rceil$ for $j=0,1$.
 Since $X$ is smooth projective, $\coh{2i-j-1}{\bar X}{i}$ has weight $-j-1\neq 0$ and so the Galois invariance   of  $\coh{2i-j-1}{\bar X}{i}$ vanishes by the Weil conjectures (see Theorem \ref{weight}); the same is true for the Galois coinvariance of it by Lemma \ref{SSdef}. 
 Hence, $\phi$ is injective.
\end{proof}

We now turn to the  injectivity assertions in \eqref{equation_roof1} and   \eqref{equation_roof2}.
To this end let $\alpha\in\urcoh{2i-j}{i-j-1}{X}{i}\subset H^{2i-j}(F_{i-j-1}X,i)$ map to zero under the morphism in \eqref{equation_roof1} or \eqref{equation_roof2}. 
This means that $\alpha$ lifts to a class $\beta\in\nrcoh{2i-j}{i-j}{X}{i}$ and the image 
$$
\phi(\beta)\in\nrcoh{2i-j}{i-j}{\bar X}{i}
$$
of $\beta$ maps to zero in $\nrcoh{2i-j}{i-1-j}{\bar X}{i}$. 
By the exactness of the bottom row of diagram \eqref{Diagram_roof1},  $\phi(\beta)$ lifts to an element $\delta\in\bigoplus\coh{j}{\overline{x}}{j}$: $i_*\delta=\phi (\beta)$.

We will show that $\delta$ can be replaced by a Galois-invariant class. 
The $j=0$ case is clear, because in this case the Galois action on $\bigoplus\coh{j}{\bar{x}}{j}$ has finite orbits on each element.
The $j=1$ case  follows from Corollary \ref{cor:liftforbeil}. 

We have thus reduced the problem to the situation where  $i_*\delta=\phi (\beta)$  for some $G_\F$-invariant class $\delta$.
Since $\delta$ is $G_\F$-invariant, it lifts to $\bigoplus\coh{1}{{x}}{1}$ and so a simple diagram chase in \eqref{Diagram_roof1} shows that we can choose the class $\beta$ above in such a way that $\phi(\beta)=0$.
The above claim then implies that $\beta=0$ and so $\alpha=0$, as we want.
\end{proof}

\begin{cor}\label{kercyc}
Let $X$ be a smooth projective variety over the finite field $\F$. 
Then there are isomorphisms
$$
\ker(\cl^{i}_{X})\simeq\urcoh{2i-1}{i-2}{\bar X}{\ql(i)}^{G_\F}
\ \ \ \text{and}\ \ \ 
\ker(\cl^{i}_{\bar X})\simeq\varinjlim \urcoh{2i-1}{i-2}{\bar X}{\ql(i)}^{G_{\F^{\prime}}},
$$
where the direct limit runs through all the finite field extensions $\F^{\prime}/\F$ with $\F^{\prime}\subseteq\bar \F$.
\end{cor}
\begin{proof}
We first note that $\coh{2i-1}{X}{\ql(i)}=0$, by the Hochschild--Serre sequence \eqref{HSexact}, because $\coh{2i-1}{\Xbar}{\ql(i)}$ and $\coh{2i-2}{\Xbar}{\ql(i)}$ have no Galois invariance and hence also no coinvariance, see Theorem \ref{smoothweight} and Lemma  \ref{coh_injection}.
By Theorem \ref{kercyc_1}, we conclude that $\ker(\cl^{i}_{X})\simeq\urcoh{2i-1}{i-2}{X}{\ql(i)}$.  
The first assertion in the corollary then follows from Proposition \ref{roof1}.

It remains to prove the second assertion.
Since any element of $\chow{i}{\bar X}$ is in the image of $\chow{i}{X_{\F^{\prime}}}\rightarrow\chow{i}{\bar X}$ for some finite extension $\F^{\prime}/\F$ with $\F^{\prime}\subseteq\bar \F$, it follows that $\ker(\cl^{i}_{\bar X})=\varinjlim \ker(\cl^{i}_{X_{\F^{\prime}}})$. 
Hence, the second assertion of the corollary follows from the first after taking direct limits.
\end{proof}

\subsection{Refined unramified cohomology and cycle conjectures over finite fields}
The next theorem expresses the validity of the Tate, Beilinson, and $1$-semi-simplicity conjectures in terms of the Galois invariance of the refined unramified cohomology. See Definition \ref{unramcoh} for the definition of the refined unramified cohomology groups.

\begin{theorem}\label{nrtatebeil}
Let $X$ be a smooth projective variety over the finite field $\F$. 
Let $G_\F$ denote the absolute Galois group of $\F$. 
Then the following are true:
\begin{enumerate}
    \item\label{nrtatebeil_1} For all $i\geq 1$, $\urcoh{2i}{i-1}{\bar X}{\ql(i)}^{G_\F}=0$ if and only if the Tate and $1$-semi-simplicity conjectures in degree $i$ hold for $X$.
    \item For all $i\geq 2$, $\urcoh{2i-1}{i-2}{\bar X}{\ql(i)}^{G_\F}=0$ if and only if the Beilinson conjecture in degree $i$ holds for $X$.
\end{enumerate}
\end{theorem}

\begin{proof}
The second item is an immediate consequence of Corollary \ref{kercyc}.

For the first assertion, consider the diagram 
\begin{equation}\label{Diagram_4.1}
    \hspace*{-2cm} 
    \begin{tikzpicture}
    \matrix[matrix of math nodes,column sep={100pt,between origins},row
    sep={70pt,between origins},nodes={asymmetrical rectangle}] (s)
    {
        &|[name=1]| \nrcoh{2i-1}{i-1}{X}{\ql(i)} &|[name=2]| \underset{x\in X^{(i)}}{\bigoplus}\ql[x] &|[name=3]| \coh{2i}{X}{\ql(i)} &|[name=4]| \nrcoh{2i}{i-1}{X}{\ql(i)} \\
        &|[name=A]| \nrcoh{2i-1}{i-1}{\bar X}{\ql(i)} &|[name=B]| \underset{\bar{x}\in \bar X^{(i)}}{\bigoplus}\ql[\bar{x}] &|[name=C]| \coh{2i}{\Xbar}{\ql(i)} &|[name=D]| \nrcoh{2i}{i-1}{\Xbar}{\ql(i)}, \\
    };
    \draw[->]
                (1) edge (A)
                (2) edge (B)
                (3) edge (C)
                (4) edge (D)
                (1) edge (2)
                (2) edge node [above] {$\iota_\ast$} (3)
                (3) edge (4)
                (A) edge (B)
                (B) edge node [above] {$\iota_\ast$} (C)
                (C) edge (D)
    ;
    \end{tikzpicture}
\end{equation}
where the vertical arrows are those from the exact sequence (\ref{HSnrexact}). Here we used Lemma \ref{ceilinglemma} and \ref{bmtoet} to identify $\nrcoh{2i}{i}{\Xbar}{\ql(i)}$ with $\coh{2i}{\Xbar}{\ql(i)}$ and $\nrcoh{2i}{i}{X}{\ql(i)}$ with $\coh{2i}{X}{\ql(i)}$. 

By Lemma \ref{coh_injection} and by the exactness of \eqref{HSexact}, $\coh{2i}{\bar X}{\ql(i)}^{G_\F}$ is isomorphic to $\coh{2i}{X}{\ql(i)}$. It follows that the Tate conjecture  for $X$ in degree $i$ is equivalent to the surjectivity of $\iota_\ast$ in the top row of  \eqref{Diagram_4.1}, which is equivalent to the vanishing of $\urcoh{2i}{i-1}{X}{\ql(i)}$ by the exactness of that row.

Assume that the Tate and $1$-semi-simplicity conjectures in degree $i$ hold for $X$. Then, the vanishing of $\urcoh{2i}{i-1}{X}{\ql(i)}$ implies the vanishing of $\urcoh{2i}{i-1}{\Xbar}{\ql(i)}^{G_\F}$ by Proposition \ref{roof1}.

Conversely, the vanishing of $\urcoh{2i}{i-1}{\Xbar}{\ql(i)}^{G_\F}$ implies the vanishing of $\urcoh{2i}{i-1}{X}{\ql(i)}$ by Proposition \ref{roof1}, and hence the Tate conjecture in degree $i$ for $X$. For the $1$-semi-simplicity conjecture for $X$, consider the following exact sequence
$$0\rightarrow T\rightarrow\coh{2i}{\Xbar}{\ql(i)}\rightarrow\urcoh{2i}{i-1}{\Xbar}{\ql(i)}\rightarrow 0,$$
which is cut out from the bottom row of Diagram (\ref{Diagram_4.1}). Here, $T$ is the image of $\cl_{\bar X}^{i}$.
By Lemma \ref{SScycle}, $G_\F$ acts $1$-semi-simply on $T$.
Since $\urcoh{2i}{i-1}{\Xbar}{\ql(i)}^{G_\F}$ vanishes, it follows from Lemma \ref{SStrans} that $\coh{2i}{\Xbar}{\ql(i)}$ is $1$-semi-simple as well.
\end{proof}

\begin{lemma}\label{roof2}
Let $X$ be a smooth projective variety over the finite field $\F$.
Then,
\begin{enumerate}
    \item for all $i\geq 1$,
    \begin{equation}\label{equation_roof2_1}
    \urcoh{2i}{i-1}{\Xbar}{\ql(i)}^{G_\F}=\nrcoh{2i}{i-1}{\Xbar}{\ql(i)}^{G_\F}.
\end{equation}
    \item for all $i\geq 2$, $\bigtate{1}{\dim X-i+1}{\F}$ implies
    \begin{equation}\label{equation_roof2_2}
    \urcoh{2i-1}{i-2}{\Xbar}{\ql(i)}^{G_\F}=\nrcoh{2i-1}{i-2}{\Xbar}{\ql(i)}^{G_\F}.
\end{equation}    
\end{enumerate}
\end{lemma}

\begin{proof}
Note first that the left-hand side of both \eqref{equation_roof2_1} and \eqref{equation_roof2_2} is contained in the right-hand side by the definition of the refined unramified cohomology, see Definition \ref{unramcoh}.
For the other inclusion, let  $j\in\{0, 1\}$.
Then we have to show that any class in $\nrcoh{2i-j}{i-j-1}{\Xbar}{\ql(i)}^{G_\F}$ is unramified, i.e.\ in the image of $\nrcoh{2i-j}{i-j}{\Xbar}{\ql(i)}$.
To see this, consider  the diagram \eqref{Diagram_roof1} with exact rows above and assume that the 
  right-most vertical arrow, namely the map
\begin{align} \label{eq:most-vertical-map in diagr roof1}
 \underset{x\in X^{(i-j)}}{\bigoplus} H^{j+1}(x,j)\ \longrightarrow \underset{ x\in  X^{(i-j)}}{\bigoplus}H^{j+1}(\bar{x},j) ,
\end{align}
is the zero map, where as before $\bar x$ denotes the base change of the point $x\in X^{(i-j)}$ to $\bar \F$,  i.e.\ the Galois orbit in $\bar X^{(i-j)}$ that corresponds to $x\in X^{(i-j)}$.
By \eqref{HSnrexact}, this is equivalent to saying that the right-hand side in \eqref{eq:most-vertical-map in diagr roof1} has no Galois-invariants.
Under this assumption, 
 any $\alpha\in\nrcoh{2i-j}{i-1-j}{\Xbar}{\ql(i)}^{G_\F}$ lifts to a class in $\nrcoh{2i-j}{i-j}{\Xbar}{\ql(i)}$ by the exactness of the bottom row in  \eqref{Diagram_roof1}. 
 
To conclude the argument, it thus suffices to show that \eqref{eq:most-vertical-map in diagr roof1} is zero.
If $j=0$,  this is clear because then the weights of $\coh{1}{\overline{x}}{\ql}$ are at least $1$ by Corollary \ref{F0weight}. 
If $j=1$, we argue in the following way: By Proposition \ref{prop:Remark_after_conjectures}, the assumption $\bigtate{1}{\dim X-i+1}{\F}$ implies $\bigSS{1}{\dim X-i+1}{\F}$.
By Theorem \ref{nrtatebeil}(\textit{\ref{nrtatebeil_1}}) and the first item of Lemma \ref{roof2}, proven above, $\bigtate{1}{\dim X-i+1}{\F}$ and $\bigSS{1}{\dim X-i+1}{\F}$ imply that $\nrcoh{2}{0}{\overline{Y}}{\ql(1)}$ has no Galois invariance for any smooth projective $\F$-variety $Y$ of dimension $\dim X-i+1$. But this implies that $\nrcoh{2}{0}{\overline{Y}}{\ql(1)}$ has no Galois invariance for any projective $\F$-variety $Y$ of dimension $\dim X-i+1$ by Corollary \ref{cor:F0vanish}. 
It  thus follows from \eqref{HSnrexact} that \eqref{eq:most-vertical-map in diagr roof1} is the zero map, as we want.
\end{proof}

\subsection{Passing from $F_jX$ to $F_{j-1}X$}

The following is an important technical result that will be used in the proof of Theorem \ref{thm:induct1}.

\begin{proposition}\label{prop:roof3}
Let $X$ be a smooth projective variety of dimension $d$ over the finite field $\F$. 
Then the following holds true.
\begin{enumerate}
    \item \label{item:prop:roof3-1} If the $1$-semi-simplicity conjecture in degree $i$ holds for $X$, then for all $j$ with $1\leq j \leq i$, the morphism
    \begin{equation}\label{equation_roof3_1}
        \nrcoh{2i}{j}{\Xbar}{\ql(i)}^{G_\F}\longrightarrow\nrcoh{2i}{j-1}{\Xbar
    }{\ql(i)}^{G_\F}
    \end{equation}
    is surjective.
    \item  \label{item:prop:roof3-2}  For all $j$ with $1\leq j \leq i-1$, if $\bigtate{i-j}{d-j}{\F}$ and $\bigSS{i-j}{d-j}{\F}$ hold,  then  the morphism \eqref{equation_roof3_1} is injective.
    \item  \label{item:prop:roof3-3}  For all $j$ with $1\leq j \leq i-2$, if $\bigtate{i-j}{d-j}{\F}$ and $\bigSS{i-j}{d-j}{\F}$ hold, then the morphism
    \begin{equation}\label{equation_roof3_2}
        \nrcoh{2i-1}{j}{\bar X}{\ql(i)}^{G_\F}\longrightarrow\nrcoh{2i-1}{j-1}{\bar X}{\ql(i)}^{G_\F}
    \end{equation}
    is surjective.
    \item  \label{item:prop:roof3-4}  For all $j$ with $1\leq j \leq i-2$, if $\bigbeil{i-j}{d-j}{\F}$,  $\bigtate{i-j-1}{d-j-1}{\F}$,  and $\bigSS{i-j-1}{d-j-1}{\F}$ hold, then the morphism \eqref{equation_roof3_2} is injective.
\end{enumerate}
\end{proposition}

\begin{proof}
For a codimension $j$ point $x\in X^{(j)}$ on $X$ we denote as before by $\bar x\subset \bar X^{(j)}$ the base change to $\bar \F$, i.e.\ the Galois orbit of $\bar \F$-points that corresponds to $x$.
For $m\in\{0, 1\}$, consider the sequence
\begin{equation}\label{sequence_roof3} 
   \underset{ x\in X^{(j)}}{\bigoplus}\coh{2i-m-2j}{\bar{x}}{i-j} \rightarrow\nrcoh{2i-m}{j}{\Xbar}{i} \rightarrow\nrcoh{2i-m}{j-1}{\Xbar}{i}\xrightarrow{\partial}\underset{ x \in X^{(j)}}{\bigoplus}\coh{2i-m-2j+1}{\bar{x}}{i-j},
\end{equation}
which is exact by Lemma \ref{exactnesslemma}. 
The $m=0$, respectively $m=1$ case, will be used for proving the first two, respectively last two, items.  

We start with the surjectivity claims; that is, with items \eqref{item:prop:roof3-1} and \eqref{item:prop:roof3-3}. 
Let 
$$
\alpha\in\nrcoh{2i-m}{j-1}{\Xbar}{\ql(i)}^{G_\F} .
$$  
It is enough to show that the residue $\partial\alpha$ of $\alpha$ vanishes. 
Indeed, if $\alpha$ has no residues, it lifts to $\nrcoh{2i-m}{j}{\Xbar}{i}$ by the exactness of \eqref{sequence_roof3}. 
It remains to show that this lift may be replaced by a Galois-invariant class. 
If $m=0$ this follows from Lemmas \ref{bmtoet} and \ref{lem:SSopencor}.
If $m=1$, it follows from Lemma \ref{lem:liftforbeil2cor}, because the assumptions $\bigtate{i-j}{d-j}{\F}$ and $\bigSS{i-j}{d-j}{\F}$ imply $\bigtate{i-j-1}{d-j-1}{\F}$ and $\bigSS{i-j-1}{d-j-1}{\F}$ by Lemma \ref{Tatedown}.
 
In order show that $\alpha$ has no residues, it suffices to show that $\coh{2i-m-2j+1}{\bar{x}}{i-j}$ has no Galois invariance for $x\in X^{(j)}$. 
If $m=0$,  $\coh{2i-m-2j+1}{\bar{x}}{i-j}$ has weights at least $1$ by Corollary \ref{F0weight}, and hence cannot have any nontrivial Galois-invariant class.
If $m=1$, we argue as follows: the assumptions $\bigtate{i-j}{d-j}{\F}$ and $\bigSS{i-j}{d-j}{\F}$ imply by Theorem \ref{nrtatebeil} and Lemma \ref{roof2}  that for any smooth projective $\F$-variety $Y$ of dimension $d-j$, the cohomology group $\nrcoh{2i-2j}{i-j-1}{\bar{Y}}{i-j}$ has trivial Galois-invariant subspace.  
Thanks to the assumption $\bigSS{i-j}{d-j}{\F}$, we may inductively use the surjectivity of  \eqref{equation_roof3_1}, proven above, to conclude that $\nrcoh{2i-2j}{0}{\overline{Y}}{i-j}$ has no nontrivial Galois-invariant classes for any smooth projective $\F$-variety $Y$ of dimension $d-j$. 
Therefore, by Corollary \ref{cor:F0vanish}, $\nrcoh{2i-2j}{0}{\overline{Y}}{i-j}^{G_\F}=0$ for any projective $k$-variety $Y$ of dimension $d-j$, as we want.

We now turn to the injectivity claims, i.e.\ to items  \eqref{item:prop:roof3-2} and \eqref{item:prop:roof3-4}. 
To this end let 
$$
\alpha\in \ker\left( \nrcoh{2i-m}{j}{\Xbar}{j}^{G_\F} \to\nrcoh{2i-m}{j-1}{\Xbar}{j} \right) .
$$
We want to show that $\alpha$ is zero.  
By exactness of  \eqref{sequence_roof3},  $\alpha$ lifts to an element 
$$
\delta\in\bigoplus_{x\in X^{j}}\coh{2i-m-2j}{\bar{x}}{i-j},\ \ \ \iota_\ast \delta=\alpha .
$$ 
We will show that $\alpha$ is zero, by proving that the lift $\delta$ may be chosen to be Galois-invariant and the Galois invariance of $\bigoplus\coh{2i-m-2j}{\bar{x}}{\ql(i-j)}$ vanishes.

Let $x_1,...,x_n\in X^{(j)}$ be the codimension-$j$ points of $X$, such that $\delta$ is supported on the base change of these points to $\bar \F$.
Since each point $x_l$ is defined over $\F$,  $G_\F$ acts summand-wise on   $\bigoplus_{l}\coh{2i-2j-m}{\overline{x}_l}{\ql(i-j)}$.

If $m=0$, the assumptions $\bigtate{i-j}{d-j}{\F}$ and $\bigSS{i-j}{d-j}{\F}$ imply, by Proposition \ref{SSsingular}, that $\cohbm{2i-2j}{\bar{X}_l}{i-j}$ is a $1$-semi-simple $G_\F$-representation for all $l$, where $\bar{X}_l=X_l\times\bar \F$ is the base change of the closure $ X_l\subset X$  of the codimension-$j$ point $x_l$. 
Hence, $\delta$ may be chosen to be Galois-invariant by Lemma \ref{lem:SSopencor}, because $\iota_\ast \delta=\alpha$ is Galois-invariant by assumptions. 
If $m=1$, it follows from Lemma \ref{lem:liftforbeil2cor} ($j=0$ case) that we may assume that $\delta$ is Galois-invariant because $\bigtate{i-j-1}{d-j-1}{\F}$ and $\bigSS{i-j-1}{d-j-1}{\F}$ hold by assumptions.

It remains to prove that the Galois invariance of 
$$
\bigoplus_{x\in X^{(j)}}\coh{2i-m-2j}{\overline{x}}{i-j}
$$ vanishes. 
We have already shown this  when $m=0$. 
The argument for $m=1$ is similar: Since $\bigtate{i-j-1}{d-j-1}{\F}$ holds and $j\leq i-2$ by assumption, $\bigtate{1}{d-i+1}{\F}$ holds by Lemma \ref{Tatedown} applied to $p=q=i+2-j$.  
By Theorem \ref{nrtatebeil} and Lemma \ref{roof2}, 
$\bigtate{1}{d-i+1}{\F}$  and $\bigbeil{i-j}{d-j}{\F}$ imply  the vanishing of the Galois invariance of $\nrcoh{2i-2j-1}{i-j-2}{\bar{Y}}{i}$ for any smooth projective $\F$-variety $Y$ of dimension $d-j$. 
Note that,  by Lemma \ref{Tatedown}, $\bigtate{i-j-1}{d-j-1}{\F}$ and 
$\bigSS{i-j-1}{d-j-1}{\F}$ 
imply $\bigtate{i-j-l}{d-j-l}{\F}$ and $\bigSS{i-j-l}{d-j-l}{\F}$ for $1\leq l\leq i-j$.
Hence,  inductively using the surjectivity of \eqref{equation_roof3_2},  
we conclude the vanishing of the Galois invariance of $\nrcoh{2i-2j-1}{0}{\bar{Y}}{i-j}$ for any smooth projective $\F$-variety $Y$ of dimension $d-j$, which in turn implies, by Corollary \ref{cor:F0vanish}, the vanishing of the Galois invariance of $\nrcoh{2i-2j-1}{0}{\bar{Y}}{i-j}$ for any quasi-projective $\F$-variety $Y$ of dimension $d-j$. 
This shows that $\bigoplus\coh{2i-2j-1}{\bar{x}}{i-j}$ has no Galois invariance, as we want.
This concludes the proof of the proposition.
\end{proof}

\begin{cor}\label{cor:Tupdown}
Let $X$ be a smooth projective variety of dimension $d$ over the finite field $\F$.
Then the following are true:
\begin{enumerate}
    \item The Tate and $1$-semi-simplicity conjectures in degree $i$ for $X$ imply the vanishing $$\nrcoh{2i}{j}{\Xbar}{\ql(i)}^{G_\F}=0$$
    for all $j$ with $0\leq j\leq i-1$. \label{item:cor:Tupdown-1}
    \item\label{Tupdown_1} If $\bigtate{i-1}{d-1}{\F }$ and $\bigSS{i-1}{d-1}{\F }$ hold, then the vanishing 
    $$
 \nrcoh{2i}{0}{\bar X}{\ql(i)}^{G_{\F }}=0
    $$
    implies the Tate and $1$-semi-simplicity conjectures in degree $i$ for $X$.\label{item:cor:Tupdown-2}
\end{enumerate}
\end{cor}

\begin{proof}
By Theorem \ref{nrtatebeil} and Lemma \ref{roof2}, the Tate and $1$-semi-simplicity conjectures in degree $i$ for $X$ implies the vanishing of the $G_\F$-invariance of $\nrcoh{2i}{i-1}{\bar X}{\ql(i)}$. 
Thanks to the $1$-semi-simplicity in degree $i$, we can inductively apply Proposition \ref{prop:roof3}\eqref{item:prop:roof3-1} to get the vanishing of the Galois invariance of $\nrcoh{2i}{j}{\bar X}{\ql(i)}$ for all $j$ with $0\leq j\leq i-1$.
This proves item \eqref{item:cor:Tupdown-1}.

By Lemma \ref{Tatedown}, $\bigtate{i-1}{d-1}{\F}$ and $\bigSS{i-1}{d-1}{\F}$ imply $\bigtate{i-j}{d-j}{\F}$ and $\bigSS{i-j}{d-j}{\F}$ for all $1\leq j\leq i$. 
This allows us to inductively apply Proposition \ref{prop:roof3}\eqref{item:prop:roof3-2} to conclude the vanishing of the Galois invariance of $\nrcoh{2i}{i-1}{\bar X}{\ql(i)}$ from the vanishing of the $G_\F$-invariance of $\nrcoh{2i}{0}{\bar X}{\ql(i)}$. 
But by Theorem \ref{nrtatebeil}, the vanishing of the Galois invariance of $\nrcoh{2i}{i-1}{\bar X}{\ql(i)}$ implies the Tate and $1$-semi-simplicity conjectures in degree $i$ for $X$. 
This proves item \eqref{item:cor:Tupdown-2}, as we want.
\end{proof}

\begin{example}
Item \eqref{item:cor:Tupdown-1} in Corollary \ref{cor:Tupdown} shows for instance that 
$$\nrcoh{2i}{0}{\CP^n_{\bar \F}}{\ql(i)}^{G_\F}=0$$
for all $i\geq 1$ and all $n$.
\end{example}
 
\begin{cor}\label{cor:Bupdown}
Let $X$ be a smooth projective variety of dimension $d$ over the finite field $\F$.  
Then the following are true:
\begin{enumerate}
    \item If $\bigtate{i-m}{d-m}{\F}$ and $\bigSS{i-m}{d-m}{\F}$ hold, then the Beilinson conjecture in degree $i$ on $X$ implies the vanishing
    $$
    \nrcoh{2i-1}{j}{\bar X}{\ql(i)}^{G_\F}=0\ \ \ \text{for all $j$ with $m-1\leq j\leq i-2$.}
    $$ 
    \item If $\bigtate{i-2}{d-2}{\F }$,  $\bigSS{i-2}{d-2}{\F }$, and $\bigbeil{i-1}{d-1}{\F }$ hold, then the vanishing
    $$
    \nrcoh{2i-1}{0}{\bar X}{\ql(i)}^{G_{\F}}=0
    $$
    implies the Beilinson conjecture in degree $i$ for $X$.
\end{enumerate}
\end{cor}
\begin{proof}
Assume that $\bigtate{i-m}{d-m}{\F}$ and $\bigSS{i-m}{d-m}{\F}$ hold. 
Then, it follows from Lemma \ref{Tatedown} that $\bigtate{i-j}{d-j}{\F}$ and $\bigSS{i-j}{d-j}{\F}$ hold for $m\leq j\leq i$ as well, in particular $\bigtate{1}{d-i+1}{\F}$ is true. 
By Lemma \ref{roof2} and Theorem \ref{nrtatebeil}, $\bigtate{1}{d-i+1}{\F}$ together with the Beilinson conjecture in degree $i$ for $X$ implies the vanishing of the $G_\F$-invariance of $\nrcoh{2i-1}{i-2}{\bar X}{\ql(j)}$. 
Now, inductively using Proposition \ref{prop:roof3}\eqref{item:prop:roof3-3}, we conclude the vanishing of the Galois invariance of $\nrcoh{2i-1}{j}{\bar X}{\ql(i)}$ for $m-1\leq j\leq i-2$. 
This proves the first item.

For the second item, by Theorem \ref{nrtatebeil}, it is sufficient to show that $\nrcoh{2i-1}{i-2}{\bar X}{\ql(i)}$ has no $G_\F$-invariance. 
By Lemma \ref{Tatedown}, $\bigbeil{i-1}{d-1}{\F}$, $\bigtate{i-2}{d-2}{\F}$,  and $\bigSS{i-2}{d-2}{\F}$ imply $\bigbeil{i-j}{d-j}{\F}$, $\bigtate{i-j-1}{d-j-1}{\F}$ and $\bigSS{i-j-1}{d-j-1}{\F}$ for all $1\leq j \leq i-1$. 
Therefore, by Proposition \ref{prop:roof3}\eqref{item:prop:roof3-4}, the morphism
$$
\nrcoh{2i-1}{j}{\bar X}{\ql(i)}^{G_\F}\rightarrow\nrcoh{2i-1}{j-1}{\bar X}{\ql(i)}^{G_\F}
$$
is injective for all $1\leq j \leq i-2$. 
Consequently, the vanishing of the Galois invariance of $\nrcoh{2i-1}{i-2}{\bar X}{\ql(i)}$ follows from that of $\nrcoh{2i-1}{0}{\bar X}{\ql(i)}$.
\end{proof}

\subsection{Proof of Theorem \ref{thm:induct1}}

\begin{proof}[Proof of Theorem \ref{thm:induct1}]
The first two items follow immediately from Corollary \ref{cor:Tupdown}. The last two items follow immediately from Corollary \ref{cor:Bupdown}.
\end{proof}

\section{Reduction to the case of the function field of $\CP^n$} \label{sec:reduction-to-Pn}

The main result of this section is the following theorem.

\begin{theorem} \label{thm:cohcrit1}
The following holds for all $d,i\geq 1$:
\begin{enumerate} 
    \item\label{cohcrit1_3}  $\bigtate{j}{d}{\F}$ and $\bigSS{j}{d}{\F}$ hold for all $j$ with $1\leq j\leq i$ if and only if
    $$
     \nrcoh{2j+1}{0}{\mathbb{P}^{d+1}_{\bar \F}}{\ql(j+1)}^{G_{\F }}=0\ \ \ \ \text{ for all $j$ with $1\leq j\leq i$.}
    $$
    \item\label{cohcrit1_4} $\bigtate{j-1}{d-1}{\F }$, $\bigSS{j-1}{d-1}{\F }$, and $\bigbeil{j}{d}{\F }$ hold   for all $j$ with $2\leq j\leq i$ if and only if
    $$
    \nrcoh{2j}{0}{\mathbb{P}^{d+1}_{\bar \F}}{\ql(j+1)}^{G_{\F }}=0 \ \ \ \ \text{ for all $j$ with $2\leq j\leq i$.}
    $$ 
\end{enumerate}
\end{theorem}


\begin{remark} \label{rem:effective-version}
Theorem \ref{thm:cohcrit1} gives effective versions of the statements in Theorems \ref{thm:main:Tate-Beilinson-ss} and \ref{thm:main:Tate-ss} in any given dimension.
For instance, the validity of the Tate and 1-semi-simplicity conjecture for surfaces over $\F$ is equivalent to $H^3(F_0\CP^3_{\bar \F},\Q_\ell(2))^{G_\F}=0$.
(Here one could drop the 1-semi-simplicity conjecture, as it is in the case of surfaces a consequence of the Tate conjecture, see Proposition \ref{prop:Remark_after_conjectures}).
Similarly, the 1-semi-simplicity and Tate conjecture for surfaces together with the Beilinson conjecture for threefolds over $\F$ is equivalent to  $H^4(F_0\CP^4_{\bar \F},\Q_\ell(3))^{G_\F}=0$.
\end{remark}

\subsection{Relation to hypersurfaces in projective space}

\begin{lemma}\label{PState_preparation}
Let $i$ and $j$ be non-negative integers.
Then the following are equivalent:
\begin{itemize}
\item $ 
\nrcoh{i}{0}{\bar X}{\ql(j)}^{G_\F}=0
$ 
for every $d$-dimensional smooth projective variety $X$ over $\F$;
\item we have $\coh{i}{\bar{x}}{\ql(j)}^{G_\F}=0$ for every $  x\in(\mathbb{P}_{\F}^{d+1})^{(1)}$.
\end{itemize}  
\end{lemma}

\begin{proof} 
By Corollary \ref{cor:F0vanish},
$  
\nrcoh{i}{0}{\bar X}{\ql(j)}^{G_\F}=0
$   for any smooth projective $\F$-variety $X$ if and only if the same vanishing holds for any quasi-projective $\F$-variety.
The lemma thus follows from the elementary fact that, since $\F$ is a perfect field,  any $d$-dimensional  $\F$-variety is birational to a hypersurface in $\CP^{d+1}_\F $,  see \cite[Proposition~I.4.9]{hartshorne}. 
\end{proof}

\begin{lemma}\label{PSbeil}
Let $i\geq 1$. If 
\begin{equation}\label{PSbeil_morphism}
\underset{ x \in (\mathbb{P}^{d+1}_{  \F})^{(1)}}{\bigoplus} \coh{2i-1}{\bar{x}}{\ql(i)} ^{G_\F}=0 ,
\end{equation} 
 then $\nrcoh{2i}{0}{\mathbb{P}^{d+1}_{\bar \F}}{\ql(i+1)}^{G_\F}=0$. 
Conversely,  \eqref{PSbeil_morphism}  holds if $\nrcoh{2i}{0}{\mathbb{P}^{d+1}_{\bar \F}}{\ql(i+1)}^{G_\F}=0$ and $\bigtate{i-1}{d-1}{\F}$ as well as $\bigSS{i-1}{d-1}{\F}$ hold.
\end{lemma}

\begin{proof} 
Let $k=\F$ or $k=\bar \F$.
Then we have
\begin{align}\label{eq:vanishing-H_nr-Pn}
H^{i}_{0,nr}(\CP^n_k,j)\cong H^{i}_{0,nr}(\Spec k,j)=0\ \ \ \text{for all $i\geq 2$},
\end{align}
where the first isomorphism follows from the stable birational invariance of unramified cohomology,  see \cite{CTO} or \cite[Corollary~1.8]{Sch-moving}, and the vanishing in question follows from the fact that finite fields have cohomological dimension $1$.
Since taking $G_\F$-invariants is a left exact functor, we thus get the following exact sequence from Lemma \ref{exactnesslemma}:
$$
0\longrightarrow H^{2i}(F_0\CP^{d+1}_{\bar \F}, i+1)^{G_\F}\longrightarrow \bigoplus_{x\in (\CP_{\F}^{d+1})^{(1)}}H^{2i-1}(\bar x,i)^{G_\F}\longrightarrow 
 H^{2i+1}(F_1\CP^{d+1}_{\bar \F}, i+1)^{G_\F} .
$$
Hence,  $\nrcoh{2i}{0}{\mathbb{P}^{d+1}_{\bar \F}}{i+1}^{G_\F}=0$ if \eqref{PSbeil_morphism} holds true.

Conversely, assume that $\bigtate{i-1}{d-1}{\F}$ as well as $\bigSS{i-1}{d-1}{\F}$ hold.
By Corollary \ref{cor:Bupdown} ($m=2$ case), $\bigtate{i-1}{d-1}{\F}$ and $\bigSS{i-1}{d-1}{\F}$ imply   $\nrcoh{2i+1}{1}{\mathbb{P}^{d+1}_{\bar \F}}{i+1}^{G_\F}=0$,  because the Beilinson conjecture holds for $\mathbb{P}^{d+1}_{\F}$. 
Hence,  $\nrcoh{2i}{0}{\mathbb{P}^{d+1}_{\bar \F}}{i+1}^{G_\F}=0$ follows from \eqref{PSbeil_morphism}. 
This concludes the proof of the lemma. 
\end{proof}

\begin{lemma}\label{PState} 
If
\begin{equation}\label{PState_morphism}
  \underset{ x \in (\mathbb{P}^{d+1}_{  \F})^{(1)}}{\bigoplus} \coh{2i}{\bar{x}}{\ql(i)}  ^{G_\F}=0 ,
\end{equation}
then $\nrcoh{2i+1}{0}{\mathbb{P}^{d+1}_{\bar \F}}{\ql(i+1)}^{G_\F}=0$.
\end{lemma}
\begin{proof}
This follows by a slight  modification of the proof of Lemma \ref{PSbeil}.
By \eqref{eq:vanishing-H_nr-Pn} and the left exactness of taking $G_\F$-invariants,  Lemma \ref{exactnesslemma} yields an exact sequence
$$
0\longrightarrow H^{2i+1}(F_0\CP^{d+1}_{\bar \F}, i+1)^{G_\F}\longrightarrow \bigoplus_{x\in (\CP_{\F}^{d+1})^{(1)}}H^{2i}(\bar x,i)^{G_\F}\longrightarrow 
 H^{2i+2}(F_1\CP^{d+1}_{\bar \F}, i+1)^{G_\F} .
$$
Since the Tate and 1-semi-simplicity conjecture hold on projective space,  
 item \eqref{induct1_1} in Theorem \ref{thm:induct1} shows that $ H^{2i+2}(F_1\CP^{d+1}_{\bar \F}, i+1)^{G_\F}=0$ vanishes, and so the result follows. 
\end{proof}

\subsection{Applications}
In this section we use the relation to hypersurfaces in projective space, established in the previous section,  to deduce various consequences from Theorem \ref{thm:induct1}.

\begin{cor}\label{induct2}
Let $d$ and $i\geq 1$ be positive integer.
Then the following are true:
\begin{enumerate}
    \item\label{induct2_1} If $\bigtate{i}{d}{\F}$ and $\bigSS{i}{d}{\F}$ hold, then
    $ 
    \nrcoh{2i+1}{0}{\mathbb{P}^{d+1}_{\bar \F}}{\ql(i+1)}^{G_\F}=0.
    $     
    \item\label{induct2_2} If $\bigtate{i-1}{d-1}{\F}$ and $\bigSS{i-1}{d-1}{\F}$ hold and
    $ 
     \nrcoh{2i+1}{0}{\mathbb{P}^{d+1}_{\bar \F}}{\ql(i+1)}^{G_{\F }}=0,
    $ 
    then $\bigtate{i}{d}{\F}$ and $\bigSS{i}{d}{\F}$ hold.
    \item\label{induct2_3} If $\bigtate{i-1}{d-1}{\F}$, $\bigSS{i-1}{d-1}{\F}$,  and $\bigbeil{i}{d}{\F}$ hold, then
    $ 
    \nrcoh{2i}{0}{\mathbb{P}^{d+1}_{\bar \F}}{\ql(i+1)}^{G_\F}=0.
    $ 
    \item\label{induct2_4} If $\bigtate{i-1}{d-1}{\F }$, $\bigSS{i-1}{d-1}{\F }$,  and $\bigbeil{i-1}{d-1}{\F }$ hold and
    $ 
    \nrcoh{2i}{0}{\mathbb{P}^{d+1}_{\bar \F}}{\ql(i+1)}^{G_{\F}}=0,
    $ 
    then $\bigbeil{i}{d}{\F}$ holds.
\end{enumerate}
\end{cor}
\begin{proof} 
We start with the first item. 
By item \eqref{induct1_1} in Theorem \ref{thm:induct1},  $\bigtate{i}{d}{\F}$ and $\bigSS{i}{d}{\F}$ imply that $\nrcoh{2i}{0}{\bar \F}{\ql(i)}^{G_\F}=0$ for every smooth projective $\F$-variety $X$. 
Therefore, 
$$
\nrcoh{2i+1}{0}{\mathbb{P}^{d+1}_{\bar \F}}{\ql(i+1)}^{G_\F}=0,
$$
by Lemmas \ref{PState_preparation} and \ref{PState}, 
as we want.

For the second item, notice that, by Lemma \ref{PState}, our assumptions imply that
$$
\underset{ x \in (\mathbb{P}^{d+1}_{  \F})^{(1)}}{\bigoplus}\coh{2i}{\bar{x}}{\ql(i)}^{G_\F}=0 .
$$ 
Hence, by Lemma \ref{PState_preparation}, $
\nrcoh{2i}{0}{\bar X}{\ql(i)}^{G_{\F }}=0
$ 
for any smooth projective $\F$-variety $X$. 
By item \eqref{induct1_2} of Theorem \ref{thm:induct1}, this implies that $\bigtate{i}{d}{\F}$ and $\bigSS{i}{d}{\F}$ hold,  as we want.

For the third item, observe that $\bigtate{i-1}{d-1}{\F}$, $\bigSS{i-1}{d-1}{\F}$, and $\bigbeil{i}{d}{\F}$ imply the vanishing
$$
\nrcoh{2i-1}{0}{\bar X}{\ql(i)}^{G_\F}=0
$$
for any smooth projective $\F$-variety $X$ by item \eqref{induct1_3} of Theorem \ref{thm:induct1}.
Therefore, by Lemmas \ref{PState_preparation} and \ref{PSbeil},
$ 
\nrcoh{2i}{0}{\mathbb{P}^{d+1}_{\bar \F}}{\ql(i+1)}^{G_\F}=0,
$ 
as we want.

For the fourth and last item, notice first that, by Lemmas  \ref{PState_preparation} and  \ref{PSbeil},  our assumptions imply that
$$
\nrcoh{2i-1}{0}{\bar X}{\ql(i)}^{G_{\F}}=0
$$
for any smooth projective $\F$-variety $X$. 
Since $\bigtate{i-1}{d-1}{\F}$ and $\bigSS{i-1}{d-1}{\F}$ imply $\bigtate{i-2}{d-2}{\F}$ and $\bigSS{i-2}{d-2}{\F}$ by Lemma \ref{Tatedown}, we conclude from item \eqref{induct1_4} in Theorem \ref{thm:induct1} that $\bigbeil{i}{d}{\F}$ holds.
\end{proof}

The following lemma will allow us to rewrite Corollary \ref{induct2} in a form that is suitable for inductive arguments.

\begin{lemma}\label{injectcup}
For all $i$,$n$, and $d$, there exists a $G_\F$-equivariant injection
$$
\nrcoh{i}{0}{\mathbb{P}^{d}_{\bar \F}}{\ql(n)}\hookrightarrow\nrcoh{i+1}{0}{\mathbb{P}^{d+1}_{\bar \F}}{\ql(n+1)}.
$$
\end{lemma}

\begin{proof}
Since $F_0X$ depends only on the generic point of $X$, we can replace projective space by affine space in the statement of the lemma.
Let $t_1,\dots ,t_{d+1}$ be the standard coordinates on $\mathbb{A}^{d+1}$. The projection to the first $d$ coordinates, $\mathbb{A}^{d+1}\xrightarrow{\pi}\mathbb{A}^{d},$
induces a pull-back map
$$
\pi^*:\nrcoh{i}{0}{\mathbb{A}^{n}}{\zl(j)}\rightarrow\nrcoh{i}{0}{\mathbb{A}^{n+1}}{\zl(j)}.
$$
Let $(t_{d+1})\in\nrcoh{1}{0}{\mathbb{A}^{d+1}}{\zl(1)}$ be the class of $t_{d+1}\in\kbar(t_1,...,t_{d+1})$, obtained via the Kummer sequence, cf.\ \cite[(P6) in Definition~4.4 and Proposition~6.6]{Sch-refined}. 
We define the following map
\begin{align*}
\nrcoh{i}{0}{\mathbb{P}^{n}}{\zl(j)}&\longrightarrow\nrcoh{i+1}{0}{\mathbb{P}^{n+1}}{\zl(j+1)}
\\
\:\alpha&\mapsto (t_{d+1})\cup \pi^*(\alpha).
\end{align*}
We claim that this map is injective.
To see this, consider the residue map 
$$
\partial_{t_{d+1}}: \nrcoh{i+1}{0}{\mathbb{P}^{n+1}}{\zl(j+1)}\longrightarrow \nrcoh{i}{0}{\mathbb{P}^{n}}{\zl(j)}
$$ 
associated to the prime divisor $\{t_{d+1}=0\}\subset \A^{d+1}$.
This map is $G_\F$-equivariant.
By the compatibility of cup-product and the residue map (see e.g.\  \cite[Lemma~2.2]{Sch-survey}), we get
$$
\partial_{t_{d+1}}((t_{d+1})\cup \pi^*(\alpha))=\partial_{t_{d+1}}(t_{d+1})\cup\alpha=\alpha.
$$ 
Altogether we have seen that $\partial_{t_{d+1}}$ yields a $G_\F$-equivariant splitting of the above cup product map.
This proves the lemma.
\end{proof}

In conjunction with the above lemma, Corollary \ref{induct2} yields the following.

\begin{cor}\label{induct3}
Thw following holds for the finite field $\F$:
    \begin{enumerate}
        \item\label{induct3_1} 
        For all $d,i\geq 1$, $\bigtate{i}{d}{\F }$ and $\bigSS{i}{d}{\F }$ hold  if and only if $\bigtate{i-1}{d-1}{\F}$ and $\bigSS{i-1}{d-1}{\F}$ hold  and
        $$
        \nrcoh{2i+1}{0}{\mathbb{P}^{d+1}_{\bar \F}}{\ql(i+1)}^{G_{\F }}=0.
        $$
        \item\label{induct3_2} 
        For all $d$ and $i\geq 2$, $\bigtate{i-1}{d-1}{\F }$, $\bigSS{i-1}{d-1}{\F }$, and $\bigbeil{i}{d}{\F }$ hold if and only if $\bigtate{i-2}{d-2}{\F }$, $\bigSS{i-2}{d-2}{\F }$, and $\bigbeil{i-1}{d-1}{\F }$ hold  and
       $$
       \nrcoh{2i}{0}{\mathbb{P}^{d+1}_{\bar \F}}{\ql(i+1)}^{G_{\F}}=0.
       $$
    \end{enumerate}
\end{cor}

\begin{proof}
    By Lemma \ref{Tatedown}, $\bigtate{i}{d}{\F}$ and $\bigSS{i}{d}{\F}$ imply $\bigtate{i-1}{d-1}{\F}$ and $\bigSS{i-1}{d-1}{\F}$. 
    Hence, the first item follows from items \eqref{induct2_1} and  \eqref{induct2_2} of Corollary \ref{induct2}.
    
   By  Lemma \ref{Tatedown} and item \eqref{induct2_3} of Corollary \ref{induct2}, 
     $\bigtate{i-1}{d-1}{\F}$, $\bigSS{i-1}{d-1}{\F}$, and $\bigbeil{i}{d}{\F}$  imply $\bigtate{i-2}{d-2}{\F}$, $\bigSS{i-2}{d-2}{\F}$, and $\bigbeil{i-1}{d-1}{\F}$ as well as the vanishing
     $$
     \nrcoh{2i}{0}{\mathbb{P}^{d+1}_{\bar \F}}{\ql(i+1)}^{G_{\F}}=0.
     $$
     For the converse, first notice that, by Lemma \ref{injectcup}, the vanishing
     \begin{equation}\label{induct3_equation1}
         \nrcoh{2i}{0}{\mathbb{P}^{d+1}_{\bar \F}}{\ql(i+1)}^{G_{\F }}=0
     \end{equation}
     implies the vanishing
     \begin{equation}\label{induct3_equation2}
         \nrcoh{2i-1}{0}{\mathbb{P}^{d}_{\bar \F}}{\ql(i)}^{G_{\F }}=0 .
     \end{equation} 
     The vanishing \eqref{induct3_equation2}, together with the assumptions $\bigtate{i-2}{d-2}{\F }$ and $\bigSS{i-2}{d-2}{\F }$ imply $\bigtate{i-1}{d-1}{\F }$ and $\bigSS{i-1}{d-1}{\F }$   by item \eqref{induct3_1} in Corollary \ref{induct3}, proven above. 
     Moreover, $\bigtate{i-1}{d-1}{\F }$ and $\bigSS{i-1}{d-1}{\F }$ together with $\bigbeil{i-1}{d-1}{\F }$ and the vanishing (\ref{induct3_equation1}) imply $\bigbeil{i}{d}{\F }$  by item \eqref{induct2_4} in Corollary \ref{induct2}.
\end{proof}

\subsection{Proof of Theorem \ref{thm:cohcrit1}}

\begin{proof}[Proof of Theorem \ref{thm:cohcrit1}]
We note that for non-negative integers $d$ and any finite field $\F$,  $\bigtate{0}{d}{\F }$, $\bigSS{0}{d}{\F }$, and $\bigbeil{1}{d}{\F}$ hold unconditionally, see Remark \ref{Remark_after_conjectures}.
\end{proof}

\subsection{Proof of Theorems  \ref{thm:main:Tate-Beilinson-ss} and  \ref{thm:main:Tate-ss}, and Corollary \ref{cor:Lefschetz}} \label{subsec:proof of thm 1 and  2}

\begin{proof}[Proof of Theorems \ref{thm:main:Tate-Beilinson-ss} and \ref{thm:main:Tate-ss}]
Recall that the 1-semi-simplicity conjecture for all smooth projective varieties over $\F$ implies the semi-simplicity conjecture for all smooth projective varieties over $\F$,  see Remark \ref{rem:1-ss_implies_ss}.
Items \eqref{cohcrit1_3} and \eqref{cohcrit1_4} of
Theorem \ref{thm:cohcrit1} therefore imply Theorems \ref{thm:main:Tate-Beilinson-ss} and \ref{thm:main:Tate-ss}, respectively. 
\end{proof}

\begin{proof}[Proof of Corollary \ref{cor:Lefschetz}]
Fix an integer $n$.
Assume that for all $d\leq n$, the Tate, Beilinson, and 1-semi-simplicity conjectures hold for cycles of codimension $i\leq \lceil d/2\rceil$ on all smooth projective varieties of dimension $d$, defined over some finite field $\F$.
This implies by Theorem \ref{thm:cohcrit1} that
$$
H^{2j}(F_0\CP^{d+1}_{\bar F},\Q_\ell(j+1))^{G_{\F}}=0
$$
for all $2\leq j\leq  \lceil d/2\rceil$.
Since the Tate and semi-simplicity conjectures hold for projective space, the above vanishing also holds true for $j=1$, see item \eqref{induct1_1} in Theorem \ref{thm:induct1}.
The statement is trivial for $j<1$ and so we find that the above vanishing holds for all $j\leq  \lceil d/2\rceil$.
Since affine varieties over algebraically closed fields have no cohomology in degrees larger than their dimension, the above vanishing is also automatic for $j> \lceil d/2\rceil$.
It then follows from Theorem \ref{thm:cohcrit1} that the Tate, Beilinson, and 1-semi-simplicity conjectures hold for cycles of arbitrary codimension on smooth projective varieties of dimension at most $n$ over $\F$. 
This proves Corollary \ref{cor:Lefschetz}.
\end{proof}

\begin{remark}
A similar argument as in Corollary \ref{cor:Lefschetz} shows that the Tate and 1-semi-simplicity conjectures in half of the degrees imply the conjectures in all degrees.
That result also follows from Deligne's hard Lefschetz theorem, but we note that our proof did not make use of this (nontrivial) result.
Moreover, a similar approach does not work to prove the assertion on the Beilinson conjecture in Corollary \ref{cor:Lefschetz}, because a hard Lefschetz theorem for Chow groups is not known in general.
\end{remark}

\section*{Acknowledgements} 
This paper contains results of the first named author's PhD thesis \cite{balkan} written under the supervision of the second named author.
We thank the referees for their helpful comments.
This project has received funding from the European Research Council (ERC) under the European Union's Horizon 2020 research and innovation programme under grant agreement No 948066 (ERC-StG RationAlgic).


\end{document}